\documentclass[10pt,a4paper,twoside,reqno]{amsart}
\usepackage{amsmath}
\usepackage{amssymb}
\usepackage{amsfonts}
\usepackage{amsthm}
\usepackage{mathrsfs}
\usepackage{dsfont}
\usepackage{mathtools}
\usepackage{bm}

\newcommand*{\mailto}[1]{\href{mailto:#1}{\nolinkurl{#1}}}

\usepackage{xcolor}
\usepackage[pagebackref=true,colorlinks=true]{hyperref}

\definecolor{darkgreen}{rgb}{0,0,0}
\definecolor{darkblue}{rgb}{0,0,0}
\definecolor{answerblue}{rgb}{0,0,0}

\hypersetup{colorlinks,breaklinks,
            linkcolor=darkblue,urlcolor=darkblue,
            anchorcolor=darkblue,citecolor=darkblue}

\newcommand{\ep}{\varepsilon}

\newcommand{\pd}{\partial}
\newcommand{\LL}{\langle}
\newcommand{\RR}{\rangle}
\newcommand{\Ex}{\mathbb{E}}
\renewcommand{\d}{\mathrm{d}}

\newcommand{\supp}{\mathrm{supp}}
\newcommand{\loc}{\mathrm{loc}}

\makeatletter
\newsavebox{\@brx}
\newcommand{\LLL}[1][]{\savebox{\@brx}{\(\m@th{#1\langle}\)}%
  \mathopen{\copy\@brx\kern-0.5\wd\@brx\usebox{\@brx}}}
\newcommand{\RRR}[1][]{\savebox{\@brx}{\(\m@th{#1\rangle}\)}%
  \mathclose{\copy\@brx\kern-0.5\wd\@brx\usebox{\@brx}}}
\makeatother

\newcommand{\weak}{\rightharpoonup}

\newcommand{\Div}{\operatorname{div}}

\newcommand{\R}{\mathbb{R}}
\newcommand{\T}{\mathbb{T}}
\newcommand{\N}{\mathbb{N}}

\newcommand{\U}{\mathfrak{U}}
\newcommand{\abs}[1]{\left | #1 \right |}
\newcommand{\norm}[1]{\left\| #1 \right\|}
\newcommand{\bk}[1]{ \left(  #1 \right)}

\newcommand{\one}[1]{\mathds{1}_{#1}}

\newcommand{\todelta}{\xrightarrow{\delta\downarrow 0}}
\newcommand{\torho}{\xrightarrow{\rho\downarrow 0}}
\newcommand{\ton}{\xrightarrow{n\uparrow \infty}}
\newcommand{\tonweak}{\xrightharpoonup{n\uparrow \infty}}

\makeatletter
\newcommand{\LLl}[1][]{\savebox{\@brx}{\(\m@th{#1\langle}\)}
  \mathopen{\copy\@brx\mkern2mu\kern-0.9\wd\@brx\usebox{\@brx}}}
\newcommand{\RRr}[1][]{\savebox{\@brx}{\(\m@th{#1\rangle}\)}
  \mathclose{\copy\@brx\mkern2mu\kern-0.9\wd\@brx\usebox{\@brx}}}
\makeatother

\makeatletter
\@namedef{subjclassname@2020}{
  \textup{2020} Mathematics Subject Classification}
\makeatother

\theoremstyle{plain}
\newtheorem{theorem}{Theorem}[section]
\newtheorem{lem}[theorem]{Lemma}
\newtheorem{remark}[theorem]{Remark}

\numberwithin{equation}{section}

\allowdisplaybreaks
\author[Karlsen]{Kenneth H. Karlsen}
\address[Kenneth H. Karlsen]{Department of Mathematics\\
  University of Oslo\\
  P.O. Box 1053 Blindern\\
  NO-0316 Oslo\\ Norway}
\email{\mailto{kennethk@math.uio.no}}

\author[Pang]{Peter H.C. Pang}
\address[Peter H.C. Pang]{School of Mathematical Sciences\\
  University of Nottingham Ningbo China\\
  199 Taikang East Road\\
  Ningbo 315100\\ China}
\email{\mailto{peter.pang@nottingham.edu.cn}}
\hypersetup{pdfauthor={Kenneth H. Karlsen and Peter H.C. Pang}}

\title[convergence of stochastic integrals]
{Convergence of stochastic integrals with 
	applications to transport equations 
	and conservation laws with noise}



\thanks{This work is funded by the Research Council of Norway
under projects 351123 (NASTRAN) and 301538 (INICE)}

\date{28 October 2026}

\begin{document}

\begin{abstract}
Convergence of stochastic integrals driven by 
Wiener processes $W_n$, with $W_n \to W$ almost surely in $C_t$, 
is crucial in analyzing SPDEs. Our focus is on the convergence of the 
form $\int_0^T V_n\, \mathrm{d} W_n \to \int_0^T V\, \mathrm{d} W$, 
where $\{V_n\}$ is bounded in $L^p(\Omega \times [0,T];X)$ for 
a Banach space $X$ and some finite $p > 2$. This is challenging 
when $V_n$ converges to $V$ \textit{weakly} in the temporal variable. 
We supply convergence results to handle stochastic integral 
limits when strong temporal convergence is lacking. 
A key tool is a uniform mean $L^1$ time translation estimate on $V_n$, 
an estimate that is easily verified in many SPDEs.  
However, this estimate alone does not guarantee strong  
compactness of $(\omega,t)\mapsto V_n(\omega,t)$. 
Our findings, especially pertinent to equations 
exhibiting singular behavior, are substantiated by establishing 
several stability results for stochastic transport equations
and conservation laws.
\end{abstract}

\maketitle

\setcounter{tocdepth}{1}

\tableofcontents

\section{Introduction}
\label{sec:intro}

We revisit the classical problem of proving 
convergence of stochastic 
integrals. Let $W_n$ and $W$ be Wiener 
processes, and $V_n$ and $V$ be predictable 
processes taking values in a Banach space, 
defined in respective filtrations for each $n$ 
and at the limit. Suppose $V_n \to V$ and 
$W_n \to W$ in some topologies. We shall 
consider the question of the convergence 
\begin{equation}\label{eq:ito_convergence1}
	\int_0^T V_n \,\d W_n 
	\overset{n\uparrow\infty}{\longrightarrow} 
	\int_0^T V\,\d W.
\end{equation}
The convergence of these integrals is essential 
in the stochastic compactness method, which 
is a fundamental component of many existence 
results for stochastic partial differential equations 
(SPDEs), see, e.g., \cite{Bensoussan:1995aa,
	DaPrato:2014aa,Debussche:2011aa}. 
Our goal is to refine some established results 
\cite{Bensoussan:1995aa,Debussche:2011aa} by focusing on 
scenarios where $V_n$ does not exhibit strong temporal 
$L^2$ convergence. To achieve this, we introduce 
a critical assumption: {\em a uniform 
mean $L^1$ temporal translation estimate for $V_n$}, 
see \eqref{eq:L1translate1} and \eqref{eq:Vn-assump_XI}. 
This assumption effectively compensates for the 
lack of strong convergence of $(\omega,t)\mapsto V_n(\omega,t)$. 
The proposed refinements are particularly 
relevant in the context of equations characterized by 
singular coefficients and/or solutions. 
As we will see, the $L^1$ translation estimate is straightforward 
to verify for many SPDEs. 

The classical Skorokhod approach, which is based on 
nontrivial results concerning the tightness 
of probability measures and the almost sure 
representations of random variables, can assist in confirming 
strong temporal convergence of $V_n$. Subsequently, 
Lemma 2.1 from \cite{Debussche:2011aa} can be applied 
to pass to the limit in the stochastic integrals. 
Let $\U$, $X$ be Hilbert spaces, and denote by $L_2(\U,X)$ 
the space of Hilbert--Schmidt operators from 
$\U$ to $X$. Let $\U_1$ be a Hilbert space into 
which $\U$ embeds, such that the embedding is 
Hilbert--Schmidt, and under which the cylindrical 
Wiener process becomes a Hilbert space-valued Wiener process
in the sense of \cite[Section 4.1.1]{DaPrato:2014aa}. Assuming that 
\begin{equation}\label{eq:dgt_assumptions}
	\begin{aligned}
		\text{(i)} & \quad V_n \to V 
		\text{ in probability in } 
		L^2([0,T]; L_2(\U,X)),
		\\ \text{(ii)} & \quad
		W_n \to W \text{ in probability in } 
		C([0,T];\U_1),
	\end{aligned}
\end{equation}
Lemma 2.1 of \cite{Debussche:2011aa} states that
\begin{equation}\label{eq:ito_conv-in-prob}
	\int_0^\cdot V_n \,\d W_n 
	\overset{n\uparrow\infty}{\longrightarrow} 
	\int_0^\cdot V\,\d W 
	\text{ in probability in } L^2([0,T];X).
\end{equation}
In practice, $V_n$ is often a deterministic map $G(u_n)$ 
of a solution $u_n$ to the SPDE of which the stochastic 
integral is a term. 
The strong temporal convergence 
assumption \eqref{eq:dgt_assumptions}-(i) 
can often be obtained via the Skorokhod approach, 
if uniform statistical bounds can be obtained for 
$u_n$ in spaces like $L^p([0,T];B_{1})\cap W^{s,p}([0,T];B_{-1})$, 
where $B_{-1}$ and $B_1$ are reflexive Banach 
spaces for which there is a Banach space $B_0$ giving 
$B_1 \Subset B_0 \subset B_{-1}$, and 
$p\in (1,\infty)$, $s \in (0,1)$ (see, e.g., \cite[Theorem 2.1]{Flandoli:1995}). 

The ``martingale identification" approach offers an alternative to 
passing to the limit in SPDE sequences, as detailed in 
\cite[Pages 229--231]{DaPrato:2014aa} (for example). 
This approach, differing from the one based 
on \cite[Lemma 2.1]{Debussche:2011aa}, bypasses convergence 
theorems for stochastic integrals. Instead, it focuses on 
computing the limits of the deterministic integral terms in the equations and 
the quadratic variation of the stochastic integrals, leveraging 
a representation theorem for martingales to establish 
the existence of a probabilistic weak solution. However, a martingale 
representation theorem is not always available in the SPDE context. 
A more recent method \cite{Brzezniak:2011aa,Ondrejat:2010aa} 
utilizes only basic martingale and stochastic integral properties,  
avoiding the martingale representation theorem. This method is 
particularly effective in ``quasi-Polish" spaces like $C([0,T];X-w)$ 
(see \cite{Jakubowski:1997aa}) and has 
been much applied in subsequent research. 
Nonetheless, while these spaces adopt a 
weak spatial topology, their temporal topology remains strong 
(see also Remark \ref{rem:strongweak} herein).

While the Skorokhod approach is widely adopted, 
it can be technically challenging and often necessitates 
detailed and lengthy proofs of convergence. 
Moreover, there are examples 
where the practicality of this approach 
might be less evident. This may occur with SPDEs that 
are too singular, hindering solutions from being 
uniformly bounded or converging in the more 
traditional spaces aforementioned. 
For some examples, see 
\cite{GHKP-inviscid,Karlsen:2023aa}. 

The key observation of our paper is that intricate 
Skorokhod-type approaches can be bypassed when the primary goal is simply 
to pass to the limit in the governing SPDEs. This will be achieved through 
refinements of \cite[Lemma 2.1]{Debussche:2011aa}, using a 
mean $L^1$ temporal translation estimate---see \eqref{eq:L1translate1} 
and \eqref{eq:Vn-assump_XI}---as a substitute for the 
strong temporal convergence 
condition \eqref{eq:dgt_assumptions}-(i).

To provide context for our results, we will 
explore two examples. Variations of these examples will be examined 
in detail in Sections $\ref{sec:ste}$ and $\ref{sec:cl}$, where they will 
serve as applications of our findings. We begin by considering 
a sequence of transport (continuity) equations 
characterized by multiplicative 
$\R^k$-valued Brownian noise:
\begin{align}\label{eq:example_spde}
	\d u_n + {\rm div}\,(b u_n) \, \d t = f\, \d t  
	+\sigma(u_n)\, \d W_n, \,\,\,\, u_n(0) =u_{0,n},
\end{align}
where the equations are interpreted in the weak sense with 
respect to the spatial variable $x$ and are evaluated pointwise 
with respect to the temporal variable $t$, being 
formulated as It{\^o} integral equations. 

The strong stability of weak solutions to deterministic 
transport equations was first addressed in the seminal work 
of DiPerna and Lions \cite{DiPerna:1989aa}.  
A natural question is whether strong stability 
results {\`a} la \cite{DiPerna:1989aa} can be established for 
stochastic transport equations such as \eqref{eq:example_spde}.
To keep the presentation simple and focused on 
the core issue, we will assume that the ``data" 
$(b, \Div b, f,u_{0,n})$ of \eqref{eq:example_spde} 
possess sufficiently high integrability. See \cite{GHKP-inviscid} 
for a more ``singular" example. We assume 
that $u_{0,n}$ converges strongly to $u_0$ in $L^p_{\omega,x}$ 
for some $p > 2$. The primary point of departure is 
the assumption that the solutions $u_n$ of the 
stochastic transport equations \eqref{eq:example_spde} 
converge weakly in $L^p_{\omega, t, x}$ to a limit $u$.

In the deterministic scenario (where $\sigma\equiv 0$), the definition 
of weak convergence would directly allow us to formulate a 
transport equation for the limit object $u$. This forms the starting point 
for any strong stability analysis. However, this trivial step becomes 
less straightforward in the stochastic case (even when the 
noise coefficient $\sigma$ is a linear function). 
If $\sigma$ is globally Lipschitz, then the $L^p$ bound on $u_n$ 
allows us to infer that there exists a non-relabelled subsequence such 
that $\sigma(u_n) \rightharpoonup \overline{\sigma}$ in $L^p_{\omega,t,x}$, 
converging weakly to some limit $\overline{\sigma}$. 
However, even when simply attempting to confirm that the weak limit 
$u$ solves the limiting equation (replacing $\sigma(u)$ with 
$\overline{\sigma}$), it becomes crucial to understand the effects of 
this weak convergence on the behavior of the 
associated stochastic integrals. 

As in \cite{DiPerna:1989aa}, our aim extends beyond merely 
establishing weak convergence; we seek to demonstrate that $u_n$ 
converges strongly in $L^q_{\omega,t,x}$ for every $1\le q<p$,
thereby ensuring that $\overline{\sigma} = \sigma(u)$. 
To achieve this transition from weak to strong convergence, we employ a 
well-known propagation of compactness strategy. This involves 
formulating equations for nonlinear compositions $\eta(u_n)$ of $u_n$, as well 
as for their respective weak limits $\overline{\eta}$. 
This leads us to stochastic integrals of the form
\begin{equation}\label{eq:intro-stoch-int}
	\int_0^T V_n \,\d W_n,  \quad 
	V_n:=\int \varphi(x)\vartheta(u_n)\, \d x, 
	\quad \varphi \in C^\infty_c,
\end{equation}
where $\vartheta(u)$ is a nonlinear function given by 
$\vartheta(u)=\eta'(u)\sigma(u)$. 
From the uniform $L^p$ bound on $u_n$, we may assume that 
$\vartheta(u_n)$ converges weakly in $L^{p_\vartheta}_{\omega, t, x}$, 
for some exponent $p_\vartheta$ dependent on $p$ and the growth 
of $\vartheta$. Consequently, our analysis is directed towards 
examining the convergence of the stochastic integrals 
\eqref{eq:intro-stoch-int}, under the condition of weak convergence 
$V_n \rightharpoonup V$ in the space $L^{p_\vartheta}_{\omega,t}$. 
Even when we replace $\sigma(u_n)$ with a strongly converging 
additive coefficient $\sigma_n$, proving convergence is still problematic. 
The issue lies with the nonlinear term $\vartheta(u_n) 
= \eta'(u_n) \sigma_n$, which remains weakly convergent.

Our second example entails a sequence of stochastic conservation 
laws of the following form:
\begin{align}\label{eq:SCL}
	d u_n + \nabla \cdot F_n(u_n)\, \d t 
	=  \sigma_{n}(u_n)\,\d W_n.
\end{align}
In the so-called kinetic formulation \cite{Lions:1994qy}, adapted 
to the stochastic setting in \cite{Debussche:2010fk}, 
the subgraphs $\chi_n(t,x, \xi) = \one{\{\xi < u_n(t,x)\}}$ of $u_n$, 
where $u_n$ are entropy solutions in the sense of Kru{\v{z}}kov, 
satisfy the following kinetic equations:

\begin{equation}\label{eq:example_spde2}
	\begin{aligned}
		\d \chi_n& + F_n'(\xi) \cdot \nabla_x \chi_n\, \d t
		\mathbin{+} \sigma_{n}(\xi) \,\pd_\xi \chi_n \,\d W_n
		\\ & \qquad 
		-\frac12\pd_{\xi}\bk{ \abs{\sigma_{n}(\xi)}^2 \pd_{\xi}\chi_n}\, \d t
		= \pd_{\xi} m_n \, \d t,
		\quad \chi_n(0) = \chi_{n,0},
	\end{aligned}
\end{equation}

where the kinetic defects $m_n$ are 
random variables taking values in the 
space of measures. The objective is to 
demonstrate that an $L^\infty_{\omega,t,x,\xi}$ 
weak-$\star$ limit $\chi$ of $\chi_n$ serves as 
a weak solution to the kinetic equation in its limiting form. 
This equation is formally deduced by 
omitting the $n$ subscripts in 
\eqref{eq:example_spde2}. Again this 
necessitates the examination of stochastic 
integrals, where the integrands exhibit 
weak convergence in $(\omega,t)$. 
This issue, when the Wiener process 
$W_n$ varies with the index $n$, is 
relevant in stochastic conservation laws 
with discontinuous flux, where there is 
a need to establish strong compactness 
through microlocal defect measures 
\cite{Karlsen:2023aa}. Furthermore, this 
problem has been encountered in the 
exploration of numerical approximations 
for stochastic conservation laws 
\cite{Dotti:2018aa,Dotti:2020}.

Motivated by these two examples, 
we are therefore particularly interested 
in the general situation described by 
\eqref{eq:ito_convergence1} when 
there is a lack of temporal compactness, i.e., 
when there is no strong topology 
serving the temporal variable in which 
the convergence $V_n \to V$ occurs.  

Without an assumption such as \eqref{eq:dgt_assumptions}-(i), 
it is still feasible to use the SPDEs to derive 
a {\em mean $L^1$-translation estimate 
in the time variable $t$}, weak in the spatial variable. 
With $v_n = \sigma(u_n)$, where $u_n$ solves 
\eqref{eq:example_spde}, we have that for any $\beta \in C^\infty_{c}$, 
\begin{align}\label{eq:L1translate1}
	\Ex \int_h^T \abs{\,\int \beta(x) 
	\bk{ v_n(t,x)-v_n(t - h,x)} \,\d x}\,\d t 
	= o_{h\downarrow 0}(1), 
\end{align}
uniformly in $n$. Similarly, for $v_n = \sigma_n(\xi) \,\pd_{\xi } \chi_n$ 
in \eqref{eq:example_spde2}, we have 
\eqref{eq:L1translate1} with the spatial 
variable $x$ replaced by the spatio-kinetic variable $(x,\xi)$.
Without $\omega$-dependence in $v_n$, 
the estimate \eqref{eq:L1translate1} implies 
strong $L^1$ compactness of 
$t\mapsto V_n=\int \beta(x) v_n(t) \,\d x$ (see \eqref{eq:intro-stoch-int} 
with $\eta'=1$) via the Kolmogorov--Riesz compactness theorem. 
Of course with $\omega$ present, \eqref{eq:L1translate1} does 
not imply strong compactness of $(\omega,t)\mapsto V_n(\omega,t)$.

The key point here is that the mean $L^1$ temporal 
translation estimate \eqref{eq:L1translate1} is readily 
available for many SPDEs. For \eqref{eq:example_spde} 
and \eqref{eq:example_spde2}, this can be shown 
by deriving the equation for $\d v_n$ from the equation for $\d u_n$ 
or $\d \chi_n$ and integrating temporally 
over $[t - h, t]$ and then once again over $[h,T]$ 
after taking absolute values (see Lemma 
\ref{thm:timetranslation} and Sections 
\ref{sec:ste} and \ref{sec:cl}). This double 
integration in time heuristically recovers 
some ``temporal compactness" in 
the sense of \eqref{eq:L1translate1}.
The estimate \eqref{eq:L1translate1} remains valid even 
when the source term $f$ in \eqref{eq:example_spde} is 
an $L^1$ function, or in cases where it is a measure, such as the 
defect term $\pd_{\xi} m_n$ in \eqref{eq:example_spde2} (see 
\cite{GHKP-inviscid,Karlsen:2023aa}).

Now assume we have the $L^1$ translation estimate as 
specified in \eqref{eq:L1translate1}, coupled with the weak 
convergence $v_n \rightharpoonup v$ in $L^p_{\omega,t,x}$ 
for some $p > 2$. The latter implies that for $\varphi \in C^\infty_{c}$, 
the processes $V_n=\int \varphi v_n\,\d x$ converge weakly in 
$L^p(\Omega \times [0,T])$. 
Our primary result (Theorem \ref{thm:main_x1}) 
then implies the convergence
\begin{align}\label{eq:windowshop}
	\int_0^T \int \varphi(x) v_n\,\d x\, \d W_n 
	\tonweak \int_0^T \int \varphi(x) v \,\d x\,\d W
	\quad \text{in $L^2(\Omega)$}. 
\end{align}
Note that the mode of convergence in 
\eqref{eq:windowshop} is {\em weak} in $L^2(\Omega)$. 
This is often sufficient when studying SPDEs, as weak convergence in 
$L^2(\Omega)=L^2(\Omega,\mathcal{F},\mathbb{P})$ 
allows us to integrate a sequence of SPDEs with 
respect to $\mathbb{P}$ against any element of 
$L^2(\Omega)$ and conclude that the limit 
SPDE holds $\mathbb{P}$--a.s., which 
is generally the desired outcome.
We will also consider the stronger assumption 
$v_n \rightharpoonup v$ a.s.~in $L^p_{t,x}$, 
in which case the convergence in \eqref{eq:windowshop} 
becomes strong in $L^2(\Omega)$.  
We will demonstrate  \eqref{eq:windowshop} by 
carefully modifying the proof of  
\cite[Lemma 2.1]{Debussche:2011aa}.

It is important to distinguish the uniformity 
in $n$ required in \eqref{eq:L1translate1} from the 
following strictly stronger condition which we do 
{\it not} assume:
\begin{align}\label{eq:L1translate2_strong}
	\Ex \sup_{n}\int_h^T \abs{\,\int \beta(x) 
	\bk{ v_n(t,x)-v_n(t - h,x)} \,\d x}\,\d t 
	= o_{h\downarrow 0}(1).
\end{align}
The supremum inside the expectation makes this a stronger
translation assumption. Neither form of temporal translation control
alone excludes oscillations in $\omega$. The additional almost sure
weak temporal convergence discussed in Section~\ref{sec:nonreflexive}
yields strong convergence for fixed spatial pairings.

The remaining part of the paper is organised as follows: 
In Section \ref{sec:mainresult}, we prove 
our main result (Theorem \ref{thm:main_x1}). 
This is done in the setting of Hilbert space-valued 
stochastic integrals, using finite dimensional noise 
for clarity of presentation. 
The technical machinery necessary to transform 
these results into an infinite dimensional context 
can be found in \cite[Chapter 9]{BanachII:2017}.
 In Section \ref{sec:nonreflexive}, 
we provide an example and adapt Theorem 
\ref{thm:main_x1} slightly to integrals taking 
values in a non-reflexive Banach space. 
In the last two sections, we give applications of our convergence 
theorems to stochastic transport equations (Section \ref{sec:ste}) 
and stochastic conservation laws (Section \ref{sec:cl}).

\section{Main convergence result}\label{sec:mainresult}
In this section, we will prove our primary convergence theorem. 

\begin{theorem}[Weak $L^2(\Omega)$ convergence of stochastic integrals]
\label{thm:main_x1}
Fix a probability space 
$\bigl(\Omega,\mathcal{F},\mathbb{P}\bigr)$. 
Let  $X$ be a real, separable Hilbert space. For each $n \in \N$, let  
$W_n$ 
be an $\R^k$-valued Wiener process  on the 
filtered probability space $\bigl(\Omega,\mathcal{F},
\{\mathcal{F}^n_t\}_{t \in [0,T]},\mathbb{P}\bigr)$, 
and $W$ an $\R^k$-valued Wiener process on $\bigl(\Omega,\mathcal{F},
\{\mathcal{F}_t\}_{t\in [0,T]},\mathbb{P}\bigr)$.

Fix $p > 2$. For each $n \in \mathbb{N}$, let 
$V_n$ be an $X^{m \times k}$-valued $\{\mathcal{F}^n_t\}$-predictable 
process, and $V$ be an $X^{m \times k}$-valued 
$\{\mathcal{F}_t\}$-predictable process. 
Suppose
\begin{equation}\label{eq:V_n_to_V_wk}
	\begin{split}
		& V_n\tonweak V \quad 
		\text{in $L^p(\Omega \times [0,T];X^{m \times k})$}.
	\end{split}
\end{equation}
We further require that 
\begin{equation}\label{eq:Wn-assump_X}
	W_n \ton W \quad \text{in $C([0,T];\R^k)$, 
	$\mathbb{P}$--a.s.}
\end{equation}

Suppose for a given $\beta \in X^*$ that the following 
mean $L^1$ temporal translation estimate holds:
\begin{align}\label{eq:Vn-assump_XI}
	\Ex \int_h^T \abs{
	\LL \beta , V_n\RR_{X^*,X}(t) - \LL \beta, V_n\RR_{X^*,X}(t - h)}\,\d t
	\xrightarrow{ h \downarrow 0} 0	
	\quad\text{uniformly in $n$.}
\end{align}
Then for any $t \in [0,T]$, 
$$
\int_0^t \bigl\LL \beta,V_n \bigr\RR_{X^*,X}  \, \d W_n
\tonweak \int_0^t \LL \beta, V \RR_{X^*,X} \, \d W
\quad \text{in $L^2(\Omega;\R^m)$}.
$$
\end{theorem}

\begin{proof}
We carry out the proof for $t = T$ only, 
remarking that each step is valid for 
any $t \in [0,T]$ in place of $T$.  
Without loss of generality, we identify 
$X$ with its dual $X^*$ and 
shall often drop the subscript on its 
angled duality bracket, which is applied 
component-wise to $V_n$ and $V$. 
We also use single line delimiters (absolute value delimiters) 
for the Euclidean norm of any finite dimension-valued object. 
We begin by noting that the weak 
convergence assumption \eqref{eq:V_n_to_V_wk}, 
implies 
\begin{align}\label{eq:V_n_to_V_wk2}
\LL \beta, V_n \RR \tonweak \LL \beta, V \RR 
\quad \text{in $L^p(\Omega \times [0,T];\R^{m \times k})$}.
\end{align}
Moreover, we get the uniform bound:
\begin{equation}\label{eq:beta_V_pairbnd}
\begin{aligned}
\norm{\LL \beta, V_n\RR}_{L^{p}(\Omega\times [0,T];\R^{m \times k})}^{p}
& = \Ex \int_0^T 
	\abs{\LL \beta, V_n(t)\RR}^p\,\d t \\
&\le \norm{\beta}_{X}^{p}
	 \Ex \int_0^T 
	  \norm{V_n(t)}_{X^{m \times k}}^{p}\,\d t\\
 & =  \norm{\beta}_{X}^{p}
\Ex \norm{V_n}_{L^p([0,T];X^{m \times k})}^{p}\lesssim 1.
\end{aligned}
\end{equation}
This uniform bound in turn gives us 
\begin{equation}\label{eq:Lp_w_L2_t_H0-new}
\Ex\norm{\LL \beta,V\RR}_{L^{p}([0,T];\R^{m \times k})}^p 
	\lesssim 1.
\end{equation}
Obviously, \eqref{eq:beta_V_pairbnd} and 
\eqref{eq:Lp_w_L2_t_H0-new} also imply 
similar uniform bounds for 
$\LL \beta, V_n \RR$ and $\LL \beta, V \RR$ in 
$L^{p_\omega}(\Omega; L^{p_t}([0,T];\R^{m \times k}))$ with 
$p_\omega, p_t \le p $. We shall often 
use these bounds with $p_t = 2$. 
%


We divide the remainder of the proof into five steps.

\noindent \textit{1. Temporal regularisation.}

For $\rho > 0$, let $\mathcal{R}(t)\ge 0$ be 
a smooth function supported on $\R_{\ge 0}$ 
such that $\int_0^\infty \mathcal{R}(t)\,\d t = 1$. 
Set $\mathcal{R}_\rho(t) 
	= \rho^{-1} \mathcal{R}(t/\rho)$. 
From the properties of $\mathcal{R}_\rho(t)$, 
\begin{align}\label{eq:R_delta_int_to1}
\forall  \delta > 0, \qquad 
\int_0^\delta \mathcal{R}_\rho(t) \,\d t 
	\xrightarrow{\rho \downarrow 0} 1.
\end{align}

We will also use $\mathcal{R}_\rho $ to 
denote the (one-sided) temporal 
regularisation operator 
\begin{align}\label{eq:R_mollification}
	\mathcal{R}_\rho f(t) 
	= \int_0^t \mathcal{R}_\rho(t - s) f(s) \,\d s, 
	\quad f\in L^1([0,T]),\quad \rho>0,
\end{align}
for $t\in [0,T]$. 
Set $\widetilde{\mathcal{R}}_\rho(t)
:=\mathcal{R}_\rho(-t)$ and also 
\begin{equation*}
	\widetilde{\mathcal{R}}_\rho f(s):=
	\int_s^T \widetilde{\mathcal{R}}_\rho(s-t)f(t)\, \d t,
	\quad s\in [0,T].
\end{equation*} 
Observe that for $f,g\in L^1([0,T])$, 
\begin{equation*}
	\int_0^T \mathcal{R}_\rho f(t)\,g(t)\,\d t
	= \int_0^T f(t)\, \widetilde{\mathcal{R}}_\rho g(t)\,\d t,
\end{equation*}
which follows from Fubini's theorem. 
We can also derive 
\begin{equation}\label{eq:Rrho-adjoint-new}
	\int_0^T \pd_t\mathcal{R}_\rho f(t)\,g(t)\,\d t
	=-\int_0^T f(t)\, 
	\pd_t\widetilde{\mathcal{R}}_\rho g(t)\,\d t.
\end{equation}
By standard convolution arguments, 
\begin{equation}\label{eq:conv-prop}
	\begin{aligned}
		& \norm{\mathcal{R}_\rho f}_{L^r([0,T])}
		\le \norm{f}_{L^r([0,T])},
		\quad 
		\norm{f-\mathcal{R}_\rho f}_{L^r([0,T])}
		\torho 0, 
		\\ &
		\norm{\widetilde{\mathcal{R}}_\rho f}_{L^r([0,T])}
		\le \norm{f}_{L^r([0,T])}, 
		\quad 
		\norm{f-\widetilde{\mathcal{R}_\rho} f}_{L^r([0,T])}
		\torho 0,
	\end{aligned}
\end{equation}
for any $f\in L^r([0,T])$ with $r\in [1,\infty)$.

Following \cite{Bensoussan:1995aa}, we use $\mathcal{R}_\rho$
to mollify the integrands of our stochastic integrals 
along the temporal direction.
For any fixed $\rho > 0$, and any $\ell \in \N$, 
convolution against $\pd_t^\ell \mathcal{R}_\rho$ 
is also a bounded 
operator $L^p_t \to L^p_t$ by the 
smoothness of $\mathcal{R}_\rho$ 
and by standard convolution arguments. Therefore 
$\mathcal{R}_\rho\bigl \LL \beta, V_n \bigr\RR$
additionally lies in the space $L^p(\Omega; C^\infty([0,T];\R^{m \times k}))$
(and similarly for $\tilde{\mathcal{R}}_\rho$ 
in place of $\mathcal{R}_\rho$).

We now consider the sequence 
$\int_0^T \bigl\LL \beta, V_n \bigr\RR\,\d W_n \in L^p(\Omega;\R^m)$ and its 
proposed limit by decomposing the difference 
between any element in the sequence and the proposed 
limit as follows:
\begin{equation}\label{eq:I(n)-split}
	\begin{split}
		I(n) &=\int_0^T \bigl\LL \beta, V_n \bigr\RR  \,\d W_n 
		-\int_0^T \LL \beta, V \RR  \,\d W
		=I_1(\rho,n)+I_2(n,\rho)+I_3(\rho),
\end{split}
\end{equation}
where
\begin{align*}
	 I_1(\rho,n) 
	&=\int_0^T \bigl\LL \beta, V_n \bigr\RR  \,\d W_n
	-\int_0^T \mathcal{R}_\rho\bigl \LL \beta, V_n \bigr\RR \,\d W_n,
	\\ 
	 I_2(n,\rho)
	&= \int_0^T \mathcal{R}_\rho \bigl\LL \beta, V_n \bigr\RR \,\d W_n
	-\int_0^T  \mathcal{R}_\rho \LL \beta, V \RR \,\d W,
	\\ 
	 I_3(\rho)
	&=\int_0^T  \mathcal{R}_\rho \LL \beta, V \RR  \,\d W
	-\int_0^T \LL \beta, V \RR  \,\d W.
\end{align*}

\noindent \textit{2. The strong 
regularisation 
limit of $I_1(\rho,n)$ as $\rho\downarrow 0$.}

Write $F_n=\LL\beta,V_n\RR$ and set
$$
b_\rho(t)=1-\int_0^t\mathcal R_\rho(s)\,\d s.
$$
By the It\^o isometry and $\abs{a+b}^2\le 2\abs a^2+2\abs b^2$,
\begin{equation}\label{eq:mollification_V_n}
	\begin{aligned}
		\Ex\abs{I_1(\rho,n)}^2
		&=\Ex\int_0^T\abs{F_n-\mathcal R_\rho F_n}^2\,\d t\\
		&\le 2\Ex\int_0^T
		\abs{\int_0^t\mathcal R_\rho(s)
		\big(F_n(t)-F_n(t-s)\big)\,\d s}^2\,\d t\\
		&\quad+2\Ex\int_0^T b_\rho(t)^2\abs{F_n(t)}^2\,\d t
		=:2J_1+2J_2.
	\end{aligned}
\end{equation}
For $0<s<T$, put
$$
A_n(s)=\Ex\int_s^T\abs{F_n(t)-F_n(t-s)}\,\d t,
\qquad
D_n(s)=\Ex\int_s^T\abs{F_n(t)-F_n(t-s)}^2\,\d t.
$$
Interpolation on the product space $\Omega\times(s,T)$ gives
\begin{align*}
	D_n(s)&\le
	\left(\Ex\int_s^T\abs{F_n(t)-F_n(t-s)}^p\,\d t
	\right)^{1/(p-1)}\\
	&\qquad\qquad\times A_n(s)^{(p-2)/(p-1)}
	\le C A_n(s)^{(p-2)/(p-1)}.
\end{align*}
Thus $\sup_n D_n(s)\to0$ as $s\downarrow0$, by
\eqref{eq:Vn-assump_XI}, and $D_n(s)\le C$ uniformly in $n,s$.
Jensen's inequality, Fubini's theorem, and $s=\rho u$ yield
$$
J_1\le\int_0^\infty\mathcal R(u)
\one{\{\rho u\le T\}}D_n(\rho u)\,\d u.
$$
Since $\mathcal R\in L^1(0,\infty)$, dominated convergence,
applied to the supremum over $n$, proves that $\sup_nJ_1\to0$.

For $J_2$, let $r=2p/(p-2)$ and fix $0<\delta<T$.
Since $0\le b_\rho\le1$ and $b_\rho$ is nonincreasing,
$$
\int_0^T b_\rho(t)^r\,\d t
\le \delta+T b_\rho(\delta)^r.
$$
H\"older's inequality and \eqref{eq:beta_V_pairbnd} give
$$
J_2\le
\left(\Ex\int_0^T\abs{F_n(t)}^p\,\d t\right)^{2/p}
\left(\int_0^T b_\rho(t)^r\,\d t\right)^{2/r}
\le C\big(\delta+T b_\rho(\delta)^r\big)^{2/r}.
$$
For each fixed $\delta$, \eqref{eq:R_delta_int_to1} implies
$b_\rho(\delta)\to0$ as $\rho\downarrow0$. Consequently,
$$
\limsup_{\rho\downarrow0}\sup_n\Ex\abs{I_1(\rho,n)}^2
\le C\delta^{(p-2)/p}.
$$
Letting $\delta\downarrow0$ now proves that
$I_1(\rho,n)\to0$ in $L^2(\Omega;\R^m)$, uniformly in $n$.

\smallskip
\noindent \textit{3. The strong regularisation 
limit of $I_3(\rho)$ as $\rho \downarrow 0$.}

By \eqref{eq:Lp_w_L2_t_H0-new},  $\LL \beta, V \RR 
\in L^2([0,T];\R^{m \times k})$, $\mathbb{P}$--a.s.~(since $p >  2$). 
Hence,
\begin{align*}
	\int_0^T \abs{ \LL \beta, V \RR
	-\mathcal{R}_\rho\LL \beta, V \RR}^2 
	\,\d t \torho 0, 
	\quad \text{$\mathbb{P}$--a.s.}
\end{align*}
By H\"older's inequality and Young's convolution 
inequality on Bochner 
spaces (see, e.g., \cite[Proposition 1.2.5, 
	Lemma 1.2.30]{BanachII:2017}),
we obtain
\begin{equation}\label{eq:V_nR_rho_highermoment}
	\begin{aligned}
		&\int_0^T \abs{\LL \beta, V \RR
		-\mathcal{R}_\rho
		\LL \beta,V \RR }^2 \,\d t
		\lesssim \int_0^T \abs{\LL \beta, V \RR }^2 
		+ \abs{\mathcal{R}_\rho
		\LL \beta, V \RR }^2 \,\d t  
		\\ & \quad 
		\le \bk{1+\norm{\mathcal{R}_\rho}_{L^1([0,T])}^{2}}
		\int_0^T \abs{\LL \beta, V \RR}^2\,\d t\\
&\quad	\lesssim  \int_0^T \abs{\LL \beta, V \RR}^2\,\d t 
		 \overset{\eqref{eq:Lp_w_L2_t_H0-new}}{\in} L^1(\Omega).
	\end{aligned}
\end{equation}
Therefore, by Lebesgue's dominated convergence theorem
(in the $\omega$ variable),
\begin{align}\label{eq:mollification_V}
	\Ex \int_0^T \abs{\LL \beta, V \RR
	-\mathcal{R}_\rho \LL \beta, V \RR}^2\,\d t
	\torho 0.
\end{align}
By the It\^o isometry, the strong 
convergence $I_3(\rho)\torho 0$ in $L^2(\Omega;\R^m)$ 
then follows.

\smallskip
\noindent \textit{4. The weak limit of 
$I_2(n,\rho)$ as $n \uparrow \infty$.}

Since $\mathcal{R}_\rho \bigl\LL \beta, V_n\bigr\RR$ 
and $\mathcal{R}_\rho \LL \beta, V\RR$ are almost surely 
smooth in time (for each fixed $\rho>0$), they have zero 
quadratic variation. Moreover, their cross-variation with 
any process with finite quadratic variation must 
also be zero. This allows us to apply integration by parts 
to each integral in $I_2(n, \rho)$ to obtain:
$$
I_2(n,\rho)=I_{2,1}+I_{2,2}+I_{2,3}+I_{2,4},
$$
where
\begin{equation}\label{eq:I_2_decomp}
	\begin{aligned}
		& I_{2,1} =-\int_0^T \bk{\pd_t \mathcal{R}_\rho 
		\bigl\LL \beta, V_n \bigr\RR  
		-\pd_t\mathcal{R}_\rho \LL \beta,V \RR }\, W_n\,\d t, 
		\\ &
		I_{2,2} = \mathcal{R}_\rho
		\bigl \LL \beta,V_n-V \bigr\RR(T)\, W_n(T), 
		\\ & 
		I_{2,3}=\int_0^T \pd_t \mathcal{R}_\rho \LL \beta, V\RR 
		\,\bk{W - W_n}\,\d t,
		\\ & 
		I_{2,4}= \mathcal{R}_\rho\LL \beta, V \RR (T)\,
		\bk{ W_n(T)-W(T)}.
	\end{aligned}
\end{equation}

We aim to show that for any $Y \in L^2(\Omega)$, 
the $n\to \infty$ limit of $\Ex[Y I_{2,1}]$ is zero. 
(It is sufficient to take a scalar-valued $Y$ 
because weak convergence in $\R^m$ is 
equivalent to element-wise convergence.)
Using the property \eqref{eq:Rrho-adjoint-new}, 
we can write $I_{2,1}$ as
$$
I_{2,1} = \int_0^T \bigl\LL \beta,V_n - V\bigr\RR 
\pd_t \widetilde{\mathcal{R}}_\rho W_n \,\d t.
$$
We first recall two facts. One, using a maximal inequality 
(see, e.g., \cite[Proposition II.1.8]{Revuz:1999wi}), 
one obtains the moment estimate
$\Ex\sup_{t\in [0,T]} \abs{W_n(t)}^{q}
\lesssim_{T,{q}} 1$, for any finite ${q}$.
Therefore, by the a.s.~convergence \eqref{eq:Wn-assump_X} 
and Vitali's convergence theorem,
\begin{equation}\label{eq:Brown-cov}
	\Ex  \sup_{t \in [0,T]} 
	\abs{W(t)-W_n(t)}^{q}
	\ton 0, \quad \text{for any finite ${q}$}.
\end{equation}
Two, the product of 
weakly and strongly converging 
sequences converge weakly. 

Due to the strong convergence \eqref{eq:Brown-cov}, for 
fixed $\rho>0$, $\pd_t\widetilde{\mathcal{R}}_\rho W_n 
\ton \pd_t\widetilde{\mathcal{R}}_\rho W$ in 
$L^{q}\bigl(\Omega;C([0,T];\R^k)\bigr)$ 
for any finite ${q}$. Therefore, the product 
$Y  \pd_t\widetilde{\mathcal{R}}_\rho W_n$ 
converges strongly in 
$L^{p'}\bigl(\Omega \times [0,T];\R^k\bigr)$ 
for any $p'<2$. Choosing $p'$ as 
the H\"older conjugate of $p$, 
$p'=\frac{p  }{p-1 } < 2$, which is 
possible as $p >2$, the weak 
convergence \eqref{eq:V_n_to_V_wk2} 
implies that $\Ex \left[Y I_{2,1}\right]$ 
converges to zero as $n \to \infty$. 

We will now examine the term $I_{2,2}$  
of \eqref{eq:I_2_decomp}. 
By \eqref{eq:V_n_to_V_wk2}, the weak 
convergence of $\bigl\LL \beta,V_n-V\bigr\RR$ 
in the temporal variable $t$ implies 
pointwise convergence of the temporally 
regularised object 
$$
\mathcal{R}_\rho
	\bigl\LL \beta,V_n-V\bigr\RR \tonweak  
	0 \quad \text{in $L^p(\Omega;\R^{m \times k})$, pointwise in $t$.}
	$$
Additionally, as noted in \eqref{eq:Brown-cov}, 
$W_n$ converges strongly to 
$W$ in $L^{q}(\Omega;C([0,T];\R^k))$ 
for any finite ${q}$. Therefore, we 
can conclude that
\begin{align*}
	I_{2,2} = \mathcal{R}_\rho 
	\bigl\LL \beta,V_n-V\bigr\RR(T) 
	W_n(T)\ton 0,
\end{align*}
weakly in $L^{p - \kappa}(\Omega;\R^m)$
for any $0 < \kappa \le p - 2 $, and 
hence weakly in $L^2(\Omega;\R^m)$ since $p  >2$.

Recall the bound \eqref{eq:Lp_w_L2_t_H0-new}.
For any $2<r<p$, let $\bar{p}<\infty$ denote the
H\"older conjugate of $r/2$. By Young's convolution
inequality and H\"older's inequality,
\begin{equation}\label{eq:I_23_converge}
	\begin{aligned}
		&\Ex\abs{I_{2,3}}^2\\
		&\quad\le\norm{\pd_t\mathcal R_\rho}_{L^1(0,T)}^2
		\left(\Ex\norm{\LL\beta,V\RR}_{L^1_t}^r\right)^{2/r}
		\left(\Ex\norm{W-W_n}_{C([0,T];\R^k)}^{2\bar p}
		\right)^{1/\bar p}\\
		&\quad\overset{\eqref{eq:Lp_w_L2_t_H0-new}}{\lesssim_\rho}
		\left(\Ex\norm{W-W_n}_{C([0,T];\R^k)}^{2\bar p}
		\right)^{1/\bar p}\ton0.
	\end{aligned}
\end{equation}

Similarly, using 
\begin{align*}
	\abs{\mathcal{R}_\rho \LL \beta, V\RR (T)}^{r} 
	& =\abs{\int_0^T \mathcal{R}_\rho(T-s)
	\LL \beta, V \RR(s)\,\d s}^{r}\\
&	 \le \norm{\mathcal{R}_\rho}_{L^\infty([0,T])}^{r}
	\norm{\LL \beta,V\RR}_{L^1([0,T];\R^{m \times k})}^r,
\end{align*}
we find that
\begin{equation}\label{eq:I_24_converge}
\begin{aligned}
	\Ex \abs{I_{2,4}}^2
	&\,\, \le \norm{\mathcal{R}_\rho}_{L^\infty([0,T])}^{2}
 	\bk{\Ex \norm{\LL \beta, V \RR}_{L^1([0,T];\R^{m \times k})}^r}^{2/r}
 	\bk{\Ex \norm{W-W_n}_{C([0,T];\R^k)}^{2\bar{p}}}^{1/\bar{p}}
 	\\ & 
 	\overset{\eqref{eq:Lp_w_L2_t_H0-new}}{\lesssim_\rho}
	\bk{\Ex \norm{W-W_n}_{C([0,T];\R^k)}^{2\bar{p}}}^{1/\bar{p}}
	\ton 0.
\end{aligned}
\end{equation}
The strong convergences \eqref{eq:I_23_converge} 
and \eqref{eq:I_24_converge} 
imply that for $Y \in L^2(\Omega)$, 
$$
\Ex \big[ YI_{2,3} \big], \,\Ex \big[ YI_{2,4} \big] \ton 0.
$$
And in summary, our findings show that
$$
I_2(n,\rho) \tonweak 0 
\quad \mbox{in $L^2(\Omega;\R^m)$, for 
any fixed $\rho>0$}.
$$

\noindent \textit{5. Conclusion.}

Returning to \eqref{eq:I(n)-split}, 
testing against 
an arbitrary but fixed $Y \in L^2(\Omega)$, 
we can make 
$$
\Ex \big[Y I (n)\big]
=\Ex \Big[ Y\bk{I_1(\rho,n)
+I_2(n,\rho)+I_3(\rho)}\Big]
$$
arbitrarily small by first picking a 
sufficiently small $\rho$ (uniformly in $n$) 
and then selecting a sufficiently large $n$ 
(depending on the chosen $\rho$). More precisely, 
we pick some small $\rho_0$ such that 
for any $\rho\leq \rho_0$ the terms 
 $I_1(\rho,n)$ and $I_3(\rho)$ become small, 
uniformly in $n$. Then we select an 
integer $N=N(\rho_0, Y)$ such that for 
any $n\ge N$ the term
$I_2(n,\rho_0)$ becomes small.
\end{proof}

\begin{remark}\label{rem:counter_p}
The following examples show 
that the uniform translation assumption
\eqref{eq:Vn-assump_XI} and the restriction $p>2$ cannot, in general,
be removed separately. Both examples use the following construction. 

Let $T=1$. 
On one probability space take a standard Brownian
motion $W$ and a standard normal random variable $Z$ independent of $W$.
For any deterministic $f_n\in L^2(0,1)$ with $\norm{f_n}_{L^2}=1$, set
$$
Z_n=\int_0^1 f_n(s)\,\d W(s),
\qquad H_n(t)=\int_0^t f_n(s)\,\d s,
$$
and define
\begin{equation}\label{eq:coupled-brownian}
	W_n(t)=W(t)+H_n(t)(Z-Z_n),\qquad 0\le t\le1.
\end{equation}
This is a continuous centred Gaussian process. Since
$\Ex[W(t)Z_n]=H_n(t)$ and $\Ex Z_n^2=\Ex Z^2=1$,
$$
\Ex[W_n(s)W_n(t)]
=\min(s,t)-2H_n(s)H_n(t)+2H_n(s)H_n(t)=\min(s,t).
$$
Thus $W_n$ is Brownian with respect to its own completed natural
filtration. Use the completed natural filtration of $W$ for the limit.
All the integrands below are deterministic and hence predictable.

\smallskip
\noindent\textit{Example \ref{rem:counter_p}.1.}
Now choose $f_n(t)=\sqrt2\sin(2\pi nt)$. This sequence is uniformly
bounded and converges weakly to zero in $L^p(\Omega\times(0,1))$
for every finite $p>1$, by the Riemann--Lebesgue lemma.
Moreover, $\norm{H_n}_{C([0,1])}\le\sqrt2/(\pi n)$.
For every $\varepsilon>0$,
$$
\sum_{n=1}^\infty\mathbb P(\abs{Z_n}>\varepsilon n)
\le\varepsilon^{-2}\sum_{n=1}^\infty n^{-2}<\infty.
$$
Borel--Cantelli gives $Z_n/n\to0$ a.s., so
$\norm{W_n-W}_{C([0,1])}\to0$ a.s. 
The missing hypothesis is the 
uniform translation estimate:
with $a_n=1/(2n)$,
$$
\int_{a_n}^1\abs{f_n(t)-f_n(t-a_n)}\,\d t
=2\int_{a_n}^1\abs{f_n(t)}\,\d t
\longrightarrow\frac{4\sqrt2}{\pi}>0.
$$

Integration by parts with the smooth 
deterministic function $f_n$ gives
$$
\int_0^1 f_n\,\d W_n
=Z_n+(Z-Z_n)\int_0^1 f_n(s)^2\,\d s=Z.
$$
Consequently, the fixed test function 
$g=Z\in L^2(\Omega)$ satisfies
$$
\Ex\left[g\int_0^1f_n\,\d W_n\right]=1
\quad\hbox{for every }n,
$$
so the integrals do not converge weakly to zero.

This gives a counterexample when \eqref{eq:Vn-assump_XI} is omitted,
even for uniformly bounded integrands.

\smallskip
\noindent\textit{Example \ref{rem:counter_p}.2.} 
Use the same probability space and the construction
\eqref{eq:coupled-brownian}, now with
$$
f_n(t)=\sqrt n\,\one{(0,1/n]}(t),\qquad
H_n(t)=\sqrt n\min(t,1/n),\qquad
Z_n=\sqrt n\,W(1/n).
$$
The covariance calculation above again 
shows that $W_n$ is Brownian
in its own completed natural filtration. Furthermore,
$$
\norm{W_n-W}_{C([0,1])}
\le\frac{\abs Z}{\sqrt n}+\abs{W(1/n)}\longrightarrow0
\quad\hbox{a.s.}
$$
We have $\norm{f_n}_{L^2(0,1)}=1$ and $f_n\rightharpoonup0$ in
$L^2(\Omega\times(0,1))$. Indeed, for every $G\in L^2(\Omega\times(0,1))$,
$$
\abs{\Ex\int_0^1 G f_n\,\d t}
\le\left(\Ex\int_0^{1/n}\abs G^2\,\d t\right)^{1/2}
\longrightarrow0.
$$
Unlike the preceding example, this sequence satisfies the required
uniform mean $L^1$ translation estimate. For $0<h<1$,
$$
\int_h^1\abs{f_n(t)-f_n(t-h)}\,\d t
\le\sqrt n\min(h,1/n)\le\sqrt h.
$$
Nevertheless, the stochastic integral is exactly
$$
\int_0^1 f_n\,\d W_n=\sqrt n\,W_n(1/n)=Z.
$$
For the fixed test function $g=Z$, we again have
$\Ex[g\int_0^1 f_n\,\d W_n]=1$ for every $n$.
Thus the stochastic integrals do not converge weakly to zero.
This is a counterexample to replacing $p>2$ by $p=2$ in
Theorem~\ref{thm:main_x1}, even when the translation estimate is retained.
There is no contradiction for $p>2$, since
$\norm{f_n}_{L^p(0,1)}=n^{1/2-1/p}\to\infty$.
\end{remark}

\section{An example with a non-reflexive space \texorpdfstring{$X$}{X}}
\label{sec:nonreflexive}
This section presents a mild extension of 
Theorem \ref{thm:main_x1} to a non-reflexive 
Banach space $X$. We examine the convergence 
of stochastic integrals with integrands $V_n$ that 
take values in $X^{m \times k}=\bk{L^1(\T^d)}^{m \times k} 
= L^1(\T^d; \R^{m \times k})$, and with 
$\R^k$-valued Brownian motions $W_n$. 
These integrands satisfy uniform bounds in 
$L^p\bigl(\Omega \times [0,T]; L^1(\T^d; \R^{m \times k})\bigr)$, 
for some $p> 2$, and weakly converge 
in $L^r(\Omega \times [0,T]\times \T^d;\R^{m \times k})$, 
for some $r\ge 1$ (possibly $r < p$ or even $r < 2$). 
The proof of the convergence theorem 
below involves addressing the discrepancy 
between these two spaces (cf.~Theorem \ref{thm:main_x1}). 
This extension is useful in the analysis of 
the stochastic Camassa--Holm equation 
(see \cite{GHKP-inviscid,HKP-viscous}). 

\begin{theorem}[Weak $L^2(\Omega)$ convergence of 
	stochastic integrals, $X = L^1(\T^d)$]
\label{thm:main_x2}
Fix a probability space 
$\bigl(\Omega,\mathcal{F},\mathbb{P}\bigr)$. 
For $n \in \N$, let $W_n$ 
be an $\R^k$-valued Wiener process 
on a filtered probability space 
$\bigl(\Omega,\mathcal{F},
\{\mathcal{F}^n_t\}_{t \in [0,T]},\mathbb{P}\bigr)$, 
and $W$ be an $\R^k$-valued Wiener process on 
a filtered probability space $\bigl(\Omega,\mathcal{F},
\{\mathcal{F}_t\}_{t\in [0,T]},\mathbb{P}\bigr)$.

Fix $p>  2$ and $ r\ge 1 $. 
For $n \in \mathbb{N}$, let 
$V_n$ be an $L^1(\T^d;\R^{m \times k})$-valued 
$\{\mathcal{F}^n_t\}$-predictable 
process, and let $V$ be an $L^1(\T^d;\R^{m \times k})$-valued 
$\{\mathcal{F}_t\}$-predictable process, such that 
\begin{equation}\label{eq:Vn-assump2}
	\begin{split}
		& V_n\tonweak V \quad 
		\text{in $L^r(\Omega\times [0,T]\times \T^d;\R^{m \times k})$}, \\
		&\Ex \norm{V_n}_{L^p([0,T];L^1(\T^d;\R^{m \times k}))}^p \lesssim 1.
	\end{split}
\end{equation}
We further require that 
\begin{equation*}
	W_n \ton W \quad \text{in $C([0,T];\R^k)$, 
	$\mathbb{P}$--a.s.}
\end{equation*}

Fix $q>d/2$. Suppose that for a 
$\beta \in H^q(\T^d)$ the following mean 
$L^1$ temporal translation estimate holds:
\begin{align}\label{eq:Vn-assump_XIII}
\Ex \int_h^T \abs{\,
	\int_{\T^d} \beta(x) \big(V_n(t) - V_n(t - h)\big)\,\d x}\,\d t
	\xrightarrow{ h \downarrow 0} 0
	\quad\text{uniformly in $n$.}
\end{align}
Then for any $t \in [0,T]$, 
\begin{align}\label{eq:thm2_conv}
\int_0^t \int_{\T^d} \beta\, V_n \,\d x  \,\d W_n
\tonweak 
\int_0^t  \int_{\T^d} \beta\, V \,\d x \,\d W
\quad \text{in $L^2(\Omega;\R^m)$}.
\end{align}
\end{theorem}

\begin{proof}
We will focus on the part of the proof 
of Theorem \ref{thm:main_x1} that does 
not apply in the case where $X=L^1(\T^d)$, 
which are Steps 3 and 4.
As in Theorem \ref{thm:main_x1}, we begin by observing 
the following bounds and convergences 
on $\int_{\T^d} \beta  V_n\,\d x$ and $\int_{\T^d}\beta  V\,\d x$: 
\begin{itemize}
\item[(i)] Using the embedding 
$H^q(\T^d) \hookrightarrow L^\infty(\T^d)$, 
we have the uniform bound 
\begin{equation}\label{eq:V_n_to_V_wk3}
\begin{aligned}
\Ex \norm{\int_{\T^d}\beta  V_n\,\d x}_{L^p_t}^p 
\le \norm{\beta}_{L^\infty_x}^p 
	\Ex \norm{ V_n}_{L^p_tL^1_x}^p \lesssim 1. 
\end{aligned}
\end{equation}
This implies the bound \eqref{eq:beta_V_pairbnd}. 
We shall establish the same bound for the limit 
$\int_{\T^d}\beta  V\,\d x$ in \eqref{eq:V-in-prod-space-II} below.

\item[(ii)] By the definition of weak limits, 
\eqref{eq:Vn-assump2} implies 
$\int_{\T^d}\beta  V_n \,\d x \tonweak \int_{\T^d}\beta  V\,\d x$ in 
$L^r(\Omega \times [0,T];\R^{m \times k})$. But \eqref{eq:V_n_to_V_wk3} 
implies 
that along a subsequence, 
\begin{align}\label{eq:V_n_to_V_wk4}
\int_{\T^d}\beta  V_n \,\d x \tonweak \overline{\Phi} 
\quad \text{ in $L^p(\Omega \times [0,T];\R^{m \times k})$}
\end{align}
to a limit $\overline{\Phi}$. By the uniqueness of 
weak limits, each subsequence in fact 
tends to the same limit 
$\overline{\Phi} = \int_{\T^d} \beta V\,\d x$, 
$\mathbb{P}\otimes \d t$-a.e.

\end{itemize}

With the regularisation operator $\mathcal{R}_\rho$ given by 
\eqref{eq:R_mollification}, it is 
again possible to split the difference
$$
I(n)=\int_0^T \int_{\T^d}  \beta \, V_n\,\d x \,\d W_n 
-\int_0^T \int_{\T^d} \beta\, V\,\d x \,\d W
$$
as in \eqref{eq:I(n)-split} into 
$I_1(n,\rho)+I_2(n,\rho)+I_3(\rho)$. 
Retaining that decomposition, we will 
compute the regularisation 
limit of $I_3(\rho)$ as $\rho \downarrow 0$, 
the $\rho \downarrow 0$ limit 
of $I_1(n,\rho)$ uniformly in $n$, 
and the limit of $I_2(n,\rho)$ as $n \uparrow \infty$.

\medskip
\noindent \textit{1. The strong regularisation 
limits of $I_1(n,\rho)$ and $I_3(\rho)$ as $\rho \downarrow 0$.}

First, we embed the non-reflexive space in which 
we have uniform bounds \eqref{eq:Vn-assump2}  
into a larger reflexive space:
\begin{equation}\label{eq:larger_H_embed}
	L^p(\Omega \times [0,T]; L^1(\T^d;\R^{m \times k})) 
	\hookrightarrow
	L^p(\Omega \times [0,T]; H^{-q}(\T^d;\R^{m \times k})),
\end{equation}
for a sufficiently large $q$. Since $V_n$ 
is uniformly bounded in the space 
$L^p(\Omega \times [0,T] ; L^1(\T^d;\R^{m \times k}))$, 
we deduce that
\begin{equation}\label{eq:Vn-H-q-bound}
	\Ex \int_0^T  \sup_{\norm{\varphi}_{ H^q(\T^d;\R^{m \times k})} = 1} 
	\abs{\,\int_{\T^d} \varphi : V_n\,\d x }^p \,\d t 
	 \lesssim  \Ex \int_0^T \norm{V_n}_{L^1(\T^d;\R^{m \times k})}^p \,\d t 
	\lesssim  1,
\end{equation}
where we have used the Sobolev embedding 
$\norm{\varphi}_{L^\infty(\T^d;\R^{m\times k})} 
	\lesssim \norm{\varphi}_{H^q(\T^d;\R^{m \times k})}=1$ 
for any $q > d/2$. We also used the colon to denote 
the scalar product between the matrix-valued 
objects $\varphi$ and $V_n$. And so  $V_n$ is uniformly bounded in 
$L^p\bigl(\Omega\times [0,T];H^{-q}(\T^d;\R^{m \times k})\bigr)$.

As this is a reflexive Banach space, we can 
use a weak compactness argument to assume that
\begin{equation*}
	V_n \tonweak \overline{V}\quad 
	\mbox{in $L^p\bigl(\Omega \times [0,T]; 
		H^{-q}(\T^d;\R^{m \times k})\bigr)$},
\end{equation*} 
for some limit $\overline{V} \in L^p_{\omega, t}H^{-q}_x$.
Test both weak limits against
$\one A(\omega)\psi(t)\zeta(x)$, where $A\in\mathcal F$,
$\psi\in C^\infty([0,T])$, and
$\zeta\in C^\infty(\T^d;\R^{m\times k})$.
Both limits have equal pairings with these test functions, which belong
to both dual spaces. Varying $A$, $\psi$ and $\zeta$ gives
$\overline V=V$ as spatial distributions for 
almost every $(\omega,t)$. This argument also applies when $r=1$.

As we now know that the limit $V$ 
of \eqref{eq:Vn-assump2} satisfies 
\begin{equation*}
	V\in L^p\bigl(\Omega \times [0,T]; H^{-q}(\T^d;\R^{m \times k})\bigr), 
\end{equation*}
it follows that (see also \eqref{eq:beta_V_pairbnd} 
and \eqref{eq:Lp_w_L2_t_H0-new}),
\begin{equation}\label{eq:V-in-prod-space-II}
	\begin{aligned}
		&\Ex \int_0^T \abs{\,\int_{\T^d} \beta\, V\,\d x}^p\, \d t
		\le \norm{\beta}_{H^q_x}^p
		\Ex \norm{V}_{L^p_tH^{-q}_x}^p
		\lesssim 1,
	\end{aligned}
\end{equation} 
so in particular, $\int_{\T^d} \beta\, V\,\d x 
\in L^2(\Omega\times [0,T];\R^{m \times k})$
and hence $\in L^2([0,T];\R^{m \times k})$ almost surely. 
Along with \eqref{eq:conv-prop}, we can 
conclude that
\begin{align*}
	\int_0^T \abs{\, \int_{\T^d} \beta\, V \, \d x
	-\mathcal{R}_\rho \int_{\T^d} \beta\, V \,\d x}^2
	\,\d t \torho 0, 
	\quad \text{$\mathbb{P}$--a.s.}
\end{align*}
Moreover, proceeding as in 
\eqref{eq:V_nR_rho_highermoment},
\begin{equation*}
	\int_0^T \abs{\, \int_{\T^d} \beta\, V \, \d x
	-\mathcal{R}_\rho \int_{\T^d} \beta\, V \,\d x}^2\,\d t
	\lesssim \int_0^T \abs{\, \int_{\T^d}\beta\, V \,\d x}^2\,\d t 
	\overset{\eqref{eq:V-in-prod-space-II}}{\in} 
	L^1(\Omega).
\end{equation*}
Therefore, by Lebesgue's dominated convergence theorem,
\begin{align*}
	\Ex \int_0^T \abs{\, \int_{\T^d} \beta\, V\,\d x
	-\mathcal{R}_\rho \int_{\T^d} \beta\, V\,\d x}^2\,\d t
	\torho 0,
\end{align*}
and, via the It\^o isometry, we conclude that 
$I_3(\rho)\torho 0$ in $L^2(\Omega;\R^m)$.

The bounds \eqref{eq:V_n_to_V_wk3} 
and \eqref{eq:V-in-prod-space-II} 
also allow us to get $I_1(n,\rho)\torho 0$ 
in $L^2(\Omega;\R^m)$, uniformly in $n$,  
as in Step 2 of Theorem \ref{thm:main_x1}.

\medskip
\noindent \textit{2. The weak   
limit of $I_2(n,\rho)$ as $n \uparrow \infty$.}

In Step 4 of the proof of Theorem \ref{thm:main_x1}, 
where the convergence of $I_2$ was analysed, 
we only used the bound \eqref{eq:Lp_w_L2_t_H0-new} 
and the weak convergence \eqref{eq:V_n_to_V_wk2} 
for the integrals $\int_{\T^d} \beta V_n\,\d x$ 
and $\int_{\T^d} \beta V\,\d x$. 
These are respectively the conditions 
\eqref{eq:V-in-prod-space-II} 
and \eqref{eq:V_n_to_V_wk4}. And we get 
$I_2(n,\rho) \tonweak 0$ in $L^2(\Omega;\R^m)$ for 
any fixed $\rho>0$, thereby allowing us 
to conclude \eqref{eq:thm2_conv} as in 
Step 5 of Theorem \ref{thm:main_x1}.
\end{proof}

\begin{remark}
If $W_n=W$ for all $n\in \N$ and the integrands are predictable
in a common filtration, then a 
related convergence result for stochastic 
integrals is discussed in 
\cite[Remark 4 (iii)]{Glatt-Holtz:2008aa}. 
There it is asserted that for $V_n \weak V$ 
in $L^2(\Omega\times [0,T])$, which 
is the assumption \eqref{eq:V_n_to_V_wk} 
with $p=2$ and $X=\R$, one also has 
$ \int_0^t V_n\,\d W \weak \int_0^t V \,\ dW$ 
in $L^2(\Omega \times [0,T])$ 
(but cf. Remark \ref{rem:counter_p}). 
Here the stochastic integral is understood 
as a process instead of being evaluated 
at $t=T$ and treated as a random variable. 
A similar assertion is made and proved in 
\cite[page 25]{Flandoli:2010yq} 
(for finite dimensional noise).  
The assertions in \cite{Flandoli:2010yq,Glatt-Holtz:2008aa} 
follow from the fact that the It\^{o} 
integral is linear and continuous from 
the space of predictable $L^2(\Omega\times [0,T])$ 
processes to $L^2(\Omega\times [0,T])$, 
and is thus also weakly continuous. 
In addition, since the space of predictable 
processes is a closed subspace of 
$L^2(\Omega\times [0,T])$, it is weakly 
closed, and so the limit $V$ is predictable. 
It is not possible to use this simple 
convergence proof when the Wiener 
process and the stochastic integrand 
both depend on $n \in \mathbb{N}$, which is 
the situation covered by Theorem \ref{thm:main_x1}.

\end{remark}

\begin{remark}{
Suppose instead of weak convergence $V_n$ to 
$V$ in all three variables as in \eqref{eq:V_n_to_V_wk}, we assume
weak convergence a.s., i.e., 
\begin{equation}\label{eq:Vn-assump_XII}
\begin{aligned}
&	V_n \tonweak V \quad 
	\mbox{in $L^p( [0,T];X^{m \times k})$ a.s.} ,\\
&	\Ex \norm{V_n}_{L^p([0,T];X^{m \times k})}^p \lesssim 1, 
\end{aligned}
\end{equation}
for some $p>2$ in Theorem \ref{thm:main_x1}, 
then similar reasoning shows the convergence 
\begin{align}\label{eq:strong_convg1}
\int_0^t \LL \beta, V_n\RR \,\d W_n 
\to \int_0^t \LL \beta, V\RR \,\d W\quad \text{in} \quad L^2(\Omega;\R^m).
\end{align}
for any $t \in [0,T]$, and the specific 
$\beta \in X^*$ for which the mean 
$L^1$-translation estimate \eqref{eq:Vn-assump_XI} holds.

Similarly, if instead of \eqref{eq:Vn-assump2}, we assume the
alternative non-reflexive space condition 
\begin{align*}
	& V_n\tonweak V \quad 
	\text{in $L^r([0,T] \times \T^d;\R^{m \times k})$, a.s.},\quad 1\le r, \, p> 2,
	\\
	& \Ex \norm{V_n}_{L^p([0,T];L^1(\T^d;\R^{m \times k}))}^p 
	\lesssim 1,
\end{align*}
then again \eqref{eq:strong_convg1} holds for any $t \in [0,T]$, and the specific 
$\beta \in H^q(\T^d)$, $q > \tfrac{d}2$  for which the mean 
$L^1$-translation estimate \eqref{eq:Vn-assump_XIII} holds.

Almost sure convergence in a weak topology
often arises after an application of the Skorokhod representation 
theorem, in which convergence in law is 
converted into a.s.~convergence on a  
different probability space. 
}

\end{remark}

\begin{remark}\label{rem:strongweak}
The time translation estimate \eqref{eq:Vn-assump_XI} 
in the almost sure convergence context of 
\eqref{eq:Vn-assump_XII} brings us very 
close to strong compactness (in $\omega,t$).
In fact, if $F_n=\int_{\T^d}\beta V_n\,\d x\rightharpoonup0$
almost surely in the temporal space, under the uniform moment and
translation assumptions of the preceding remark, then we 
see immediately that the strong $L^2(\Omega)$ 
convergence \eqref{eq:strong_convg1} implies 
$$
\Ex \int_0^T \abs{F_n}^2\,\d t =
\Ex \,\Big|\,\int_0^T F_n\,\d W_n\Big|^2  \to 0.
$$

The convergence of $F_n$
in $L^2_{\omega, t}$ cannot happen in the same 
way under the weak $L^2(\Omega)$ 
convergence provided by Theorems \ref{thm:main_x1} 
and \ref{thm:main_x2}. The strong convergence 
of $\int_{\T^d} \beta V_n\,\d x$ suggests that in some 
situations, a suitable topology to consider 
the convergence of $V_n$ is in fact  in 
$L^p(\Omega \times [0,T];L^2(\T^d)-w)$, 
with the topology of strong convergence in $(\omega, t)$ and 
weak convergence in the spatial variable.
Here $p$ is chosen for the relevant application and need not
coincide with the exponent above.
That is, 
for every $\beta\in L^2(\T^d)$, 
$$
\quad \Ex \int_0^T \abs{\,\int_{\T^d} \beta(x)  
	\bk{V_n(t) - V(t)}\,\d x }^p\,\d t  \to 0.
$$

This is the $L^p$ analogue to the 
time-continuous space $C([0,T];X-w)$ 
mentioned in Section \ref{sec:intro} of our paper 
(see \cite{Brzezniak:2011aa,Ondrejat:2010aa}), 
and of the space $\mathbb{D}([0,T];L^2(\T^d)-w)$ 
which is Skorokhod (c\`adl\`ag) in time and weakly 
$L^2$ in space, used in \cite{brzezniak:2018aa}.
Strong-weak spaces of the form $L^p([0,T];X-w)$ 
featured saliently in the recent works 
\cite{GHKP-inviscid} and \cite{Karlsen:2023aa}.

\end{remark}

\section{Stochastic transport equations}
\label{sec:ste}
	
In this and the next section, we consider two applications 
of the limit theorems of Sections \ref{sec:mainresult} and \ref{sec:nonreflexive}. 
Here we establish a 
stability result for a scalar semilinear stochastic 
transport equation for which the nonlinearity 
is in the noise term. We use the term ``transport 
equation" loosely to mean either the transport 
equation or the continuity equation, and in fact 
mostly work with the latter. Our results apply 
to either type of equation with minimal technical 
modifications between the proofs, including the 
presence or absence of an additional lower 
order term. In our equations, we also incorporate a vanishing 
viscosity term to address the numerical diffusion 
commonly encountered in numerical schemes.

More precisely, we wish to develop a stochastic analogue 
of the strong stability result \cite[Theorem II.4]{DiPerna:1989aa}. 
Let $(\Omega, \mathcal{F}, \{\mathcal{F}_t\}_{t \in [0,T]}, \mathbb{P})$ 
be a filtered probability space on which $W$ 
is an $\R^k$-valued Brownian motion. 
Finite-dimensional noise is employed here for simplicity of presentation. 

We say that an adapted element $u 
\in L^p(\Omega\times [0,T] \times \T^d)$ 
is a weak solution of 
\begin{align}\label{eq:semilinear_ste}
	\d u + {\rm div}\,(b u) \,\d t 
	= f \,\d t + \ep \Delta u\,\d t 
	+ \sigma(u) \,\d W, 
	\qquad u(0) = u_{0}
\end{align}
if the equation holds a.s.~weakly in $x$ 
in the sense of It\^o, i.e., for every $\varphi \in C^2(\T^d)$, 
a.s.~for every $t \in (0,T]$,
\begin{equation}\label{eq:weaksoln_ste}
	\begin{aligned}
		&\int_{\T^d} u(t) \varphi\,\d x - \int_{\T^d} u_0 \varphi\,\d x
		-\int_0^t \int_{\T^d} \nabla \varphi \cdot b u\,\d x\,\d s\\
		& \qquad 
		= \int_0^t \int_{\T^d} f \varphi\,\d x \,\d s
		+\ep \int_0^t \int_{\T^d}  \Delta \varphi u\,\d x\,\d s
		+ \int_0^t \int_{\T^d} \sigma(u)\varphi\,\d x \,\d W.
	\end{aligned}
\end{equation}

Consider now solutions $u_n$ to 
\eqref{eq:semilinear_ste} above 
where $b$, $f$, $\sigma$, and $\ep$ are indexed 
with a parameter $n$. We assume that $u_n \weak u$ 
in $L^p(\Omega \times [0,T] \times \T^d)$ for some $p > 2$.  
The strong convergence $u_n \to u$, where $u$ 
is a weak solution of a limiting equation, will depend on 
a renormalisation procedure, in which a nonlinear 
function $\eta(u_n)$ of $u_n$ is evaluated. 
In the expression for $\d \eta(u_n)$, another nonlinear 
function ${\vartheta}(u_n) 
	:= \eta'(u_n) \sigma(u_n)$ of $u_n$ appears 
in the stochastic integral against $\d W_n$ 
(which are different for each $n$). 
Observe that even if $u_n \to u$ strongly in time and weakly 
in space (say), no strong convergence will be 
preserved {\em a priori} for the composition
${\vartheta}(u_n)$, which will only satisfy 
$L^p$ bounds in $(\omega ,t ,x)$ depending on 
its growth in terms of its argument. 
This leaves us with a {\em weak} limit 
${\vartheta}(u_n) \weak \overline{{\vartheta}}$ on
$L^p_{\omega, t, x}$ for the integrand of the 
stochastic integral. The convergence 
of these stochastic integrals then does not seem 
to follow from the standard convergence 
lemma of, e.g., \cite[Lemma 2.1]{Debussche:2011aa}
because of the lack of strong temporal compactness. 
However, we shall see from these examples that 
a temporal translation estimate of the form 
\eqref{eq:Vn-assump_XI} is often available to solutions of 
even quite singular SPDEs (Lemma \ref{thm:timetranslation}).

We shall establish the following analogue
to  \cite[Theorem II.4]{DiPerna:1989aa}:

\begin{theorem}\label{thm:ste_semilinear}
Fix $p > 2$. Set $p' = p/\bk{p - 1}$ and 
$p'' = p/\bk{p - 2}$, respectively the 
H\"older conjugates of $p$ and $p/2$. 
Let $\{u_n\}_{n = 1}^\infty$ be 
a sequence of weak solutions to 
\begin{align}\label{eq:semilinear_ste_n}
\d u_n + {\rm div}\,(b_n u_n)\,\d t
  = f_n \,\d t+ \frac1n \Delta u_n \,\d t 
  	+ \sigma(u_n)\,\d{W}_n, 
	\quad u_n(0) = u_{0,n}
\end{align}

for which $u_n\rightharpoonup u$ in 
$L^p(\Omega \times [0,T] \times \T^d)$.
Let $W_n$ and $W$ be Wiener processes with respect to
$(\mathcal F_t^n)$ and $(\mathcal F_t)$, respectively.
Assume that $u_n$ is $(\mathcal F_t^n)$-predictable, and that
$u$ and all weak limits appearing as stochastic integrands in the proof
admit $(\mathcal F_t)$-predictable representatives.
The initial data $u_{0,n}$ and $u_0$ are assumed
$\mathcal F_0^n$- and $\mathcal F_0$-measurable, respectively.
Assume also that, for each $n$,
$$
K_n:=\Ex\,\operatorname*{ess\,sup}_{0<t<T}
\norm{u_n(t)}_{L^p(\T^d)}^p<\infty.
$$
Suppose:
\begin{itemize}
\item[(i)] $\sigma \in C^{1,1}(\R;\R^k)$ is globally Lipschitz on $\R$, 
\item[(ii)]
	$\{b_n\}_{n \ge 1} \subset_b 
		L^1([0,T];W^{1,p''}(\T^d;\R^d))$, and \\
	$\{{\rm div}\,(b_n)\}_{n \ge 1}\subset_b 
		L^1([0,T];L^\infty(\T^d))$,  
\item[(iii)] $b_n \to b$ in $L^1([0,T];W^{1,p'}(\T^d;\R^d))$,\\
	${\rm div}\,(b_n) \to {\rm div}\,(b)$ in 
	$L^1([0,T];L^{p''}(\T^d))$,
\item[(iv)] $f_n \to f$ in $L^1([0,T];L^p(\T^d))$, 
$u_{0,n} \to u_0$ in $L^{p}(\Omega ;L^p(\T^d))$, and 
\item[(v)]
	$W_n \to W$ a.s.~in $C([0,T];\R^k)$.
\end{itemize}
Then $u$ is a weak solution to 
\eqref{eq:semilinear_ste} with $\ep = 0$, 
and 
$u_n \to u$ strongly in $L^q(\Omega \times [0,T] \times \T^d)$
for every $1\le q<p$. 
\end{theorem}

(We use ``$\subset_b$" to indicate a  
subset is bounded.)

\begin{remark}
We assumed weak convergence 
$u_n \rightharpoonup u$ in 
$L^{q_1}_{\omega, t} L^{q_2}_x$ in Theorem 
\ref{thm:ste_semilinear} with $q_1 = q_2 =  p$. 
The $q_1 > 2$ integrability 
in time is necessary for applying the 
convergence theorems of foregoing sections. 
The endpoint $q_2=2$ requires a separate uniform-integrability
estimate for the quadratic compositions.
See Remark \ref{rem:pf_artifact}.
For $q_1=q_2=p$, the conclusion is strong convergence in
$L^q_{\omega,t,x}$ for every $1\le q<p$.
More generally, if an analogous quadratic argument gives strong
$L^2_{\omega,t,x}$ convergence under a uniform
$L^{q_1}_{\omega,t}L^{q_2}_x$ bound with $q_1,q_2>2$, mixed-norm
interpolation yields strong convergence in
$L^{r_1}_{\omega,t}L^{r_2}_x$ for $1\le r_1<q_1$ and $1\le r_2<q_2$.

\end{remark}

To establish this theorem, we 
shall repeatedly use the following result 
to get uniform time translation 
estimates for solutions to SPDEs. 
Let $M_b$ denote the space of 
bounded (finite total variation) signed Radon measures. 
\begin{lem}\label{thm:timetranslation}
Let $(\Omega, \mathcal{F}, 
	\{\mathcal{F}^n_t\}_{t \ge 0}, \mathbb{P})$ 
be a sequence of filtered probability spaces. 
Denote multi-indices by  
$\alpha \in \mathbb{N}^d$, and let $m_F, m_G\ge 0$ 
be two integers. Suppose 
\begin{itemize}
\item[(i)] for each $\abs{\alpha} \le m_F$, 
	$\{F_n^{(\alpha)}\}_{n \ge 1} \subset_b 
	L^1(\Omega;M_b([0,T] \times \T^d))$,
\item[(ii)] for each $\abs{\alpha} \le m_G$, 
	$\{G_n^{(\alpha)}\}_{n \ge 1} \subset_b 
	L^2(\Omega \times [0,T];L^1(\T^d;\R^k))$ where 
	$G_n^{(\alpha)}$ is  $\{\mathcal{F}^n_t\}_{t \ge 0}$-predictable, 
\item[(iii)] $\{\varrho_{0,n}\}_{n \ge 1} 
	\subset_b L^1(\Omega \times \T^d)$, and
\item[(iv)] for each $n \in \N$, $W_n$ is an $\R^k$-valued 
	$\{\mathcal{F}_t^n\}_{t \ge 0}$-Wiener process.
\end{itemize}
Let $\{\varrho_n\}_{n \ge 1}$ be solutions 
to the equation 
$$
\d \varrho_n = \sum_{\abs{\alpha} \le m_F}
	\pd_x^{\alpha} F_n^{(\alpha)} \,\d t 
	+ \sum_{\abs{\alpha} \le m_G}
	\pd_x^\alpha G_n^{(\alpha)}\,\d W_n,
	$$
 i.e.,
	for any $\varphi \in C^\infty(\T^d)$, 
	and a.s.~for a.e.~$t \in [0,T]$,
\begin{equation}\label{eq:ttrans_equation}
\begin{aligned}
\int_{\T^d}  \varrho_n(t)\varphi\,\d x  
	- \int_{\T^d}  \varrho_{0,n}\varphi\,\d x 
	&= \sum_{\abs{\alpha} \le m_F}
	\bk{-1}^{\abs{\alpha}} \int_0^T\int_{\T^d}
	 \pd_x^{\alpha}\varphi \one{[0,t]}\, F_n^{(\alpha)}(\d r,\d x)\\
&\quad\,\, + \sum_{\abs{\alpha} \le m_G}
	\bk{-1}^{\abs{\alpha}}\int_0^t \int_{\T^d} 
 	\pd_x^\alpha\varphi \,  G_n^{(\alpha)}\,\d x \,\d W_n.
\end{aligned}
\end{equation}
Then for any $\psi \in C^\infty(\T^d)$, 
there is a constant $C_\psi > 0$, dependent 
on $\norm{\pd_x^\alpha\psi}_{L^\infty(\T^d)}$ 
for every $\abs{\alpha} \le m_F \vee m_G$, but 
independent of $n$, such that 
\begin{align*}
\Ex \int_h^T \abs{\,\int_{\T^d} \psi(x) 
	\big(\varrho_n(t) - \varrho_n(t - h)\big)\,\d x}\,\d t 
	\le C_\psi h^{1/2}
	\quad \text{uniformly in $n$}.
\end{align*}
\end{lem}

\begin{proof}
For a fixed test function $\varphi$, put
$$
g_n^{(\alpha)}(r)=\int_{\T^d}
\pd_x^\alpha\varphi(x)G_n^{(\alpha)}(r,x)\,\d x.
$$
Then
$$
\Ex\int_0^T\abs{g_n^{(\alpha)}(r)}^2\,\d r
\le\norm{\pd_x^\alpha\varphi}_{L^\infty_x}^2
\Ex\int_0^T\norm{G_n^{(\alpha)}(r)}_{L^1_x}^2\,\d r
\lesssim_\varphi1.
$$
Subtract \eqref{eq:ttrans_equation} at $t-h$ from the same identity
at $t$. The finite-variation terms are bounded, after integration
in $t$ and expectation, by
\begin{align*}
	&\sum_{\abs\alpha\le m_F}
	\norm{\pd_x^\alpha\varphi}_{L^\infty_x}
	\Ex\int_{[0,T]\times\T^d}
	\left(\int_h^T\one{\{t-h<r\le t\}}\,\d t\right)
	\abs{F_n^{(\alpha)}}(\d r,\d x)\\
	&\qquad\le h\sum_{\abs\alpha\le m_F}
	\norm{\pd_x^\alpha\varphi}_{L^\infty_x}
	\Ex\norm{F_n^{(\alpha)}}_{TV([0,T]\times\T^d)}
	\lesssim_\varphi h.
\end{align*}
For the martingale terms, Cauchy--Schwarz and the BDG inequality give
\begin{align*}
	\int_h^T\Ex\abs{\int_{t-h}^t
	g_n^{(\alpha)}(r)\,\d W_n(r)}\,\d t
	&\lesssim T^{1/2}\left(\int_h^T\Ex\int_{t-h}^t
	\abs{g_n^{(\alpha)}(r)}^2\,\d r\,\d t\right)^{1/2}\\
	&\lesssim (Th)^{1/2}\left(\Ex\int_0^T
	\abs{g_n^{(\alpha)}(r)}^2\,\d r\right)^{1/2}
	\lesssim_{\varphi,T}h^{1/2}.
\end{align*}
Summing the finitely many terms proves the claim, since
$h\le T^{1/2}h^{1/2}$. In particular, the spatial $L^1$ bound in (ii)
is sufficient for all applications below.
\end{proof}

\begin{remark}
Even though the statement of Lemma  
\ref{thm:timetranslation} is written for $x \in \T^d$, 
it is clear from the proof that it holds for 
$x \in \R^d$, under the qualification that  
the test functions $\varphi$ must be 
compactly supported.
\end{remark}

We will repeatedly use the following 
simple fact about weak convergence.

\begin{lem}\label{thm:wkcv_ae}
Suppose $F_n = F_n(\omega, t) \le 0$ for 
a.e.~$(\omega,t)\in \Omega \times [0,T]$, for each $n\in \N$.
Furthermore, suppose we have the weak convergence 
$F_n \rightharpoonup F$ in $L^p(\Omega \times [0,T])$, 
for some $p\in [1,\infty)$. Then $F \le 0$ a.e.~on $\Omega \times [0,T]$. 
The same statement holds with ``$\le$" 
replaced consistently by ``$\ge$" or ``$=$" throughout.
\end{lem}

\begin{proof}
Since $L^\infty(\Omega \times [0,T]) 
\subset L^{p'}(\Omega \times [0,T])$, 
where $p'$ is conjugate to $p$ (including $p'=\infty$),
by choosing a non-negative element 
$Y \in L^\infty(\Omega \times [0,T])$, 
we see immediately that 
$0 \ge \Ex \int_0^T Y F_n \,\d t \to \Ex \int_0^T Y F\,\d t$.
By the arbitrariness of $Y\ge 0$, it follows that $F \le 0$ 
a.e.~on $\Omega \times [0,T]$. 
\end{proof}

\begin{proof}[Proof of Theorem \ref{thm:ste_semilinear}]
We prove this theorem with the additional 
technical assumptions that 
\begin{align}\label{eq:sigma_technical}
u_{0,n} \to u_0\quad \text{in $L^{2p}(\Omega ;L^p(\T^d))$}, \qquad
\abs{\sigma''(v)}\lesssim1 \wedge \abs{v}^{p - 3}.
\end{align}
These can be removed --- see Remark  \ref{rem:pf_artifact}. 
Our emphasis will be on the convergence of stochastic integrals.

The strategy is to take a nonlinear function 
$\eta(v) = \frac12 v^2$, derive an inequality for
the weak limit $\overline{\eta}$ of $\eta(u_n)$ 
and to derive an equation $\eta(u)$ in order 
to compare the two quantities. 
The $\mathbb{P}\otimes \d t \otimes \d x$-a.e.~coincidence 
$\eta(u) = \overline{\eta}$ 
will imply the strong convergence $u_n \to u$. 

First we derive an equation for $\eta(u_n)$, 
whose weak limit will give an inequality for $\overline{\eta}$.
Let $J_\delta$ be a standard mollifier on 
$\T^d$. For locally integrable functions 
$g$, let $g_\delta := J_\delta * g$. By 
testing \eqref{eq:semilinear_ste_n} 
against $J_\delta$, we have the 
pointwise-in-$x$ equation
\begin{align*}
&u_{n,\delta}(t) - u_{0,n,\delta} + \int_0^t 
	{\rm div}\,(b_{n,\delta} u_{n,\delta}) \,\d s \\
& = \int_0^t f_{n,\delta} \,\d s 
	+ \frac1n \int_0^t \Delta u_{n,\delta} \,\d s 
	+\int_0^t\sigma(u_n)_\delta\,\d W_n
	+\int_0^t E^{(1)}_{n,\delta}[u_n]\,\d s.
\end{align*}
The term $E^{(1)}_{n,\delta}$ is the commutator 
$E^{(1)}_{n,\delta}[u] :={\rm div}\,(b_{n,\delta} u_\delta)
-{\rm div} J_\delta *(b_n u)$. 
For fixed $n$, the DiPerna--Lions commutator estimate
\cite[Lemma II.1]{Lio1996b}, with $1/p''+1/p=1/p'$, gives
$$
E^{(1)}_{n,\delta}[u_n]\longrightarrow0
\quad\text{in }L^p(\Omega;L^1(0,T;L^{p'}(\T^d))).
$$
Indeed, its norm is bounded by
$C\norm{b_n}_{L^1_tW^{1,p''}_x}K_n^{1/p}$, and dominated
convergence applies in $\omega$.
Thus its product with $\vartheta'(u_{n,\delta})$ tends to zero in
$L^1_{\omega,t,x}$ whenever $\vartheta'$ has at most linear
growth.

We can now apply the It\^o formula 
pointwise in $x$, to get that for any 
entropy $S=S(x,u) \in C(\T^d;C^{1,1}(\R))$, 
a.s.~for all $s,t \in [0,T]$, $s \le t$, 
\begin{equation*}
\begin{aligned}
	 &{S}(x,u_{n,\delta}(t)) -  {S}(x,u_{n, \delta}(s)) 
	+\int_s^t {S}'(x,u_{n,\delta}(r)) 
	\,{\rm div}\bk{b_{n,\delta} u_{n,\delta}}  \,\d r \\
	& =\int_s^t   {S}'(x,u_{n,\delta}(r)) f_{n,\delta} \,\d r
	+ \frac1n \int_s^t {S}'(x,u_{n,\delta}(r)) \Delta u_{n,\delta} \,\d r\\
	&\quad\,\,
	+ \frac12 \int_s^t 
	{S}''(x,u_{n,\delta}(r)) \abs{\sigma(u_n(r))_\delta}^2\,\d r  \\
	&\quad\,\,+ \int_s^t 
	 {S}'(x,u_{n,\delta}(r)) \sigma(u_n(r))_{\delta}\,\d W_n
	 +\int_s^t S'(x,u_{n,\delta}(r)) E^{(1)}_{n,\delta}[u_n]\,\d r,
\end{aligned}
\end{equation*}
where ${S}'$ and ${S}''$ refer to 
derivatives in the second argument only.
Suppose $S(x,v) = \psi(x) \vartheta(v)$, $\psi \in C^2(\T^d)$, and 
\begin{equation}\label{eq:beta_assump1}
\abs{\vartheta(v)} \lesssim 1 + \abs{v}^p, \quad 
\abs{\vartheta'(v)} \lesssim 1 + \abs{v}^{p - 1}, \quad 
\abs{\vartheta''(v)} \lesssim 1 + \abs{v}^{p - 2}.
\end{equation}
Then we can integrate in $x$ to get
\begin{equation}\label{eq:renorm_semilinear2}
\begin{aligned}
	&\int_{\T^d} \psi(x) \big(\vartheta(u_{n,\delta}(t)) 
 	-  \vartheta(u_{n, \delta}(s))\big) \,\d x\\ 
	&\quad +\int_s^t\int_{\T^d}  \psi\,{\rm div}\bk{b_{n,\delta}} 
	\bk{ \vartheta'(u_{n,\delta})u_{n,\delta} - \vartheta(u_{n,\delta})}
	-b_{n,\delta} \cdot \nabla \psi \vartheta(u_{n,\delta})\,\d x \,\d r\\
	& =\int_s^t  \int_{\T^d} \psi \vartheta'(u_{n,\delta}) f_{n,\delta} \,\d x \,\d r
	+ \int_s^t\int_{\T^d} \psi\vartheta'(u_{n,\delta}) E^{(1)}_{n,\delta}[u_n]\,\d x\,\d r\\
	&\quad + \frac1n \int_s^t\int_{\T^d} \Delta \psi \vartheta(u_{n,\delta})  
	- \psi \vartheta''(u_{n,\delta}) \abs{\nabla u_{n,\delta}}^2\,\d x \,\d r\\
	&\quad
	+ \frac12 \int_s^t  \int_{\T^d} 
	\psi\vartheta''(u_{n,\delta}) \abs{\sigma(u_n)_\delta}^2\,\d x\,\d r
	+ \int_s^t  \int_{\T^d}
	\psi \vartheta'(u_{n,\delta}) \sigma(u_n)_\delta\,\d x\,\d W_n.
\end{aligned}
\end{equation}

The global Lipschitz assumption on 
$\sigma$ implies $\abs{\sigma(v)} \lesssim 1 + \abs{v}$. 
Therefore, $\sigma(u_n) \in 
	L^p(\Omega \times [0,T] \times \T^d;\R^k)$. 
First take $\vartheta(v)=v^2/2$ and $\psi\equiv1$ in
\eqref{eq:renorm_semilinear2}. Finiteness of $K_n$, localisation,
and the commutator estimate give
$$
\frac1n\Ex\int_0^T\int_{\T^d}
\abs{\nabla u_{n,\delta}}^2\,\d x\,\d t\le C_n<\infty.
$$
Weak compactness gives $\nabla u_n\in L^2_{\omega,t,x}$ for fixed
$n$. Hence $\nabla u_{n,\delta}\to\nabla u_n$ strongly in this space.

For $R\ge1$, let $\beta_R$ be convex with
$\beta_R(0)=\beta_R'(0)=0$ and
$\beta_R''(v)=(p-1)(|v|\wedge R)^{p-2}$.
Its bounded second derivative and linearly growing first derivative
allow $\delta\downarrow0$ in \eqref{eq:renorm_semilinear2} with
$\vartheta=\beta_R$ and $\psi\equiv1$.
Uniformly in $R$,
\begin{align*}
0\le v\beta_R'(v)-\beta_R(v)&\lesssim_p\beta_R(v),
\qquad |\beta_R'(v)|^{p'}\lesssim_p\beta_R(v),\\
\beta_R''(v)|\sigma(v)|^2+|\beta_R'(v)\sigma(v)|
&\lesssim_{p,\sigma}1+\beta_R(v).
\end{align*}
H\"older's inequality bounds the forcing term by
$C_p\norm{f_n}_{L^p_x}(1+\int_{\T^d}\beta_R(u_n)\,\d x)$.
After localisation, BDG and Gronwall therefore give
$$
\Ex\,\operatorname*{ess\,sup}_{0<t<T}
\int_{\T^d}\beta_R(u_n(t))\,\d x
+\frac1n\Ex\int_0^T\int_{\T^d}
\beta_R''(u_n)|\nabla u_n|^2\,\d x\,\d t\le C,
$$
where $C$ depends only on the uniformly bounded data in (i)--(iv),
$p$ and $T$, and is independent of $K_n$, $n$ and $R$.
Since $\beta_R(v)\uparrow|v|^p/p$, Fatou's lemma yields
\begin{equation}\label{eq:p_energybound}
\begin{aligned}
\Ex\,\operatorname*{ess\,sup}_{0<t<T}
\norm{u_n(t)}_{L^p(\T^d)}^p
+\frac1n\Ex\int_0^T\int_{\T^d}
|u_n|^{p-2}|\nabla u_n|^2\,\d x\,\d t\le C.
\end{aligned}
\end{equation}
The quadratic calculation above, now using this uniform bound, also gives
$$
\frac1n\Ex\int_0^T\int_{\T^d}|\nabla u_n|^2\,\d x\,\d t\le C
\quad\text{uniformly in }n.
$$

For the remaining limit $\delta\downarrow0$ in
\eqref{eq:renorm_semilinear2}, take $\vartheta$ with bounded second
derivative, as in \eqref{eq:beta_assump2} below.
Its first derivative has at most linear growth, so the commutator
integral tends to zero.

By the standard properties of mollifiers 
and the temporal integrability on 
$\vartheta'(u_n)$ and $\sigma(u_n)$ 
implied by \eqref{eq:p_energybound}, 
\eqref{eq:beta_assump1}, and 
$\abs{\sigma(v)} \lesssim 1 + \abs{v}$, we have
\begin{align}\label{eq:temam1}
	\int_{\T^d} 
	 \psi \vartheta'(u_{n,\delta}) \sigma(u_n)_\delta\,\d x
	\todelta \int_{\T^d} 
	 \psi \vartheta'(u_{n}) \sigma(u_n)\,\d x
	 \quad\text{in $L^2([0,T];\R^k)$ a.s.},
\end{align}
so that there is strong temporal $L^2$ convergence of 
the stochastic integrands indexed by 
the mollification parameter $\delta > 0$. 
Thus, with reference to \eqref{eq:dgt_assumptions} 
-- \eqref{eq:ito_conv-in-prob}, 
the stochastic integral then tends to itself 
without the subscript $\delta$ by the 
standard convergence result of 
\cite[Lemma 2.1]{Debussche:2011aa}. 

The boundedness condition used here is
\begin{align}\label{eq:beta_assump2}
\abs{\vartheta''(v)}\lesssim 1.
\end{align}
For smooth $\vartheta$,
$\nabla u_{n,\delta}\to\nabla u_n$ in $L^2_{\omega,t,x}$ and
boundedness of $\vartheta''$ give convergence of the dissipation term
in $L^1_{\omega,t,x}$. Smooth approximation gives the
$C^{1,1}$ case. 
Finally, we get the following equation 
a.s.~for every $t, s \in [0,T]$, $s \le t$:
\begin{equation}\label{eq:renorm_semilinear}
\begin{aligned}
\int_{\T^d} \psi(x) 
	&\big(\vartheta(u_n(t)) -  \vartheta(u_{n}(s))\big) \,\d x \\
&= \sum_{j = 1}^4 \int_s^t\int_{\T^d} I_j(n) \,\d x \,\d r 
	+ \int_s^t \int_{\T^d}  I_5(n) \,\d x\,\d W_n, 
\end{aligned}
\end{equation}
where
\begin{align*} 
I_1(n) &:= -   \psi\,{\rm div}\bk{b_n} 
	\bk{ \vartheta'(u_n)u_n - \vartheta(u_n)}
	+ b_n \cdot \nabla \psi \vartheta(u_n),\\
I_2(n) &:=  \psi \vartheta'(u_n) f_n,\qquad
I_3(n) :=  \frac1n\Delta \psi \vartheta(u_n)
		- \frac1n \psi \vartheta''(u_n) \abs{\nabla u_n}^2,\\
I_4(n) &:= 
	\frac12 	\psi\vartheta''(u_n) \abs{\sigma(u_n)}^2, \qquad 
I_5(n) := 
	 \psi \vartheta'(u_n) \sigma(u_n).
\end{align*}

With $\vartheta(v) = \eta(v) = \frac12 v^2$, and $\psi \equiv 1$, 
 $I_3 \le 0$, $(\omega, t, x)$-a.e. Therefore, we find:
\begin{equation}\label{eq:eta(u_n)_ste}
\begin{aligned}
\int_{\T^d} &\eta(u_n(t)) -  \eta(u_{0,n})\,\d x 
	+\int_0^t \int_{\T^d} \bk{\eta'(u_n) u_n - \eta(u_n)}\,
		{\rm div}\bk{b_n} \,\d x \,\d s \\
& \le\int_0^t  \int_{\T^d} \eta'(u_n) f_n 
	+ \frac12 \eta''(u_n) \abs{\sigma(u_n)}^2 \,\d x \,\d s  \\
&\quad\,\,+ \int_0^t \int_{\T^d}
	 \eta'(u_n) \sigma(u_n)\,\d x \,\d W_n,\quad \text{a.s.}
\end{aligned}
\end{equation}
We will take limits of terms in this inequality 
to get an inequality for $\overline{\eta}$.

By the bound \eqref{eq:p_energybound} on $u_n$, 
\begin{align*}
\{\eta(u_n)\}, \{\eta'(u_n)u_n \}, \{\eta'(u_n)\} 
	\subset_b
	L^{p/2}(\Omega \times [0,T] \times \T^d).
\end{align*}
We deduce that along an unrelabelled subsequence, 
these three sequences tend respectively to 
weak limits $\overline{\eta}$, $\overline{\eta' u}$ 
and $\overline{\eta'}$ in 
$ L^{p/2}(\Omega \times [0,T] \times \T^d)$. 
Additionally, by the assumption (i) on $\sigma$, 
for each $j \in \{1,\ldots,k\}$, 
\begin{align}\label{eq:Lp/2_bound_ste}
\{\eta'(u_n) \sigma_j(u_n)\} , 
	\{ \eta''(u_n) \abs{\sigma(u_n)}^2\} 
	\subset_b 
	L^{p/2}(\Omega \times [0,T] \times \T^d). 
\end{align}
Let $\overline{\eta'\sigma}$ 
and $\overline{\eta''\abs{\sigma}^2}$ 
respectively be the (component-wise) weak limits of 
$\eta'(u_n) \sigma(u_n)$ and 
$\eta''(u_n) \abs{\sigma(u_n)}^2$ in this space. 
As $p > 2$, these limits are not singular measures. 

We now take a closer look at the 
limit of the martingale term, whose 
convergence shall be a result of Theorem \ref{thm:main_x2}.
Exploiting the higher moment assumption 
on the initial data $u_{0,n}$ in $L^p(\T^d)$, 
by the same localised truncation argument used for
\eqref{eq:p_energybound}, 
\begin{equation*}
\Ex \sup_{t \in [0,T]} 
	\norm{u_n(t)}_{L^p(\T^d)}^{2p}  \lesssim 1.
\end{equation*}
Hence,
\begin{equation}\label{eq:u_0_L2p_bound}
\begin{aligned}
\Ex \norm{\eta'(u_n) 	
	\sigma(u_n) }_{L^p([0,T];L^1(\T^d;\R^k))}^{p}
	&\lesssim 1 + \Ex \int_0^T \bk{\int_{\T^d} 
		\abs{u_n}^{2} \,\d x}^p \,\d t\\
	&\lesssim 
	\Ex \sup_{t \in [0,T]}\norm{u_n(t)}_{L^p(\T^d)}^{2p}
	 \lesssim 1.
\end{aligned}
\end{equation}

Turning to the $L^1$ temporal translation estimate
\eqref{eq:Vn-assump_XIII} of Theorem~\ref{thm:main_x2}, put
$\vartheta(v)=\eta'(v)\sigma(v)=v\sigma(v)$.
Choose $\chi\in C_c^\infty(\R)$ with $0\le\chi\le1$,
$\chi=1$ on $[-1,1]$, and $\supp\chi\subset[-2,2]$.
For $M\ge1$, set
$$
\vartheta_M(v)=\chi(v/M)\vartheta(v).
$$
For fixed $M$, $\vartheta_M$ and its first two derivatives are bounded.
Use \eqref{eq:renorm_semilinear} with $\psi\in C^2(\T^d)$,
$\vartheta=\vartheta_M$ and $s=t-h$; denote the resulting coefficients
by $I_j^M(n)$, $j=1,\ldots,5$.

All bounds below are uniform in $n$ for fixed $M$.
The coefficient assumptions and \eqref{eq:p_energybound} give
\begin{equation}\label{eq:ste_translation_1}
I_1^M(n), I_2^M(n)\in_b L^1(\Omega\times[0,T]\times\T^d).
\end{equation}
The quadratic dissipation bound proved above and boundedness of
$\vartheta_M''$ give
\begin{equation}\label{eq:ste_translation_3}
|I_3^M(n)|\in_b L^1(\Omega\times[0,T]\times\T^d).
\end{equation}
The same bound on $\vartheta_M''$ and the linear growth of $\sigma$
yield
\begin{equation}\label{eq:ste_translation_4}
I_4^M(n)\in_b L^1(\Omega\times[0,T]\times\T^d).
\end{equation}
Since $\vartheta_M'\sigma$ is bounded,
\begin{equation}\label{eq:ste_translation_5}
I_5^M(n)\in_b L^2(\Omega\times[0,T];L^1(\T^d;\R^k)).
\end{equation}
Lemma~\ref{thm:timetranslation} bounds the mean $L^1$
translations of $\int_{\T^d}\psi\vartheta_M(u_n)\,\d x$ by
$C_{\psi,M}h^{1/2}$.
To remove the cutoff, use $|\vartheta(v)|\lesssim1+|v|^2$ and
\eqref{eq:p_energybound}:
$$
\sup_n\Ex\int_0^T\int_{\T^d}
|\vartheta(u_n)-\vartheta_M(u_n)|\,\d x\,\d t
\le C M^{2-p}.
$$
Consequently, for the original integrand,
$$
\Ex\int_h^T\left|\int_{\T^d}\psi
\big(\vartheta(u_n(t))-\vartheta(u_n(t-h))\big)\,\d x\right|\,\d t
\le C_{\psi,M}h^{1/2}+C_\psi M^{2-p},
$$
uniformly in $n$.
First let $h\downarrow0$ with $M$ fixed, then $M\uparrow\infty$.
Since $p>2$, this proves \eqref{eq:Vn-assump_XIII}
for $\vartheta=\eta'\sigma$.
Then by Theorem \ref{thm:main_x2}
for any $t \in [0,T]$, choosing $\psi \equiv 1$, 
\begin{align*}
\int_0^t \int_{\T^d} 
		\eta'(u_n) \sigma(u_n) \,\d x \,\d W_n
		\tonweak \int_0^t \int_{\T^d} 
		\overline{\eta' \sigma} \,\d x \,\d W
		\quad \text{in $L^2(\Omega)$}.
\end{align*}

In view of the assumed strong convergences 
of $f_n$,  ${\rm div}\,(b_n)$, $u_{0,n}$, the 
remaining integrals of \eqref{eq:eta(u_n)_ste} 
converge appropriately, and by Lemma 
\ref{thm:wkcv_ae} we have the 
$\mathbb{P}\otimes \d t$-a.e.~inequality: 
\begin{equation}\label{eq:u_n_limit_equ_ste}
\begin{aligned}
&\int_{\T^d} \overline{\eta}(t) - \eta(u_0)\,\d x 
+ \int_0^t \int_{\T^d} \overline{\bk{\eta' u - \eta}} 
	\,{\rm div}\,(b) \,\d x\,\d s\\
& \le \int_0^t \int_{\T^d} \overline{\eta'} f\,\d x \,\d s
	+ \frac12\int_0^t \int_{\T^d} 
		\overline{\eta'' \abs{\sigma}^2} \,\d x \,\d s
	+ \int_0^t \int_{\T^d} 
		\overline{\eta' \sigma} \,\d x \,\d W.
\end{aligned}
\end{equation}

Next,  we derive an equation for $\eta(u)$, 
for which we first need an equation for $u$. 
To do this we take a weak limit of \eqref{eq:semilinear_ste_n}. 
On the assumptions of the 
theorem, except for the stochastic integral, 
the appropriate weak limit holds with 
$(u, u_0, b, f)$ in place of $(u_n, u_{0,n}, b_n, f_n)$ 
in \eqref{eq:semilinear_ste_n}. 

Since $u_n \rightharpoonup u$ in 
$L^p(\Omega \times[0,T] \times \T^d)$ 
and hence is bounded in $L^p(\Omega \times [0,T];L^2(\T^d))$, 
and $\sigma$ has at most linear growth, we may assume that 
$$
\sigma(u_n) \rightharpoonup \overline{\sigma} \quad
\text{in $L^p(\Omega\times [0,T];L^2(\T^d))$}.
$$
We shall invoke Theorem \ref{thm:main_x1} 
to prove the convergence 
of the stochastic integral. 

Recall that $\sigma$ satisfies 
$$
\abs{\sigma(v)} \lesssim 1 + \abs{v}, 
\quad 
\abs{\sigma'(v)} \lesssim 1, 
\quad
\abs{\sigma''(v)} \lesssim 1 \wedge |v|^{p - 3},
$$
Set $s=t-h$, $\vartheta=\sigma$ and $\psi=\varphi$ in
\eqref{eq:renorm_semilinear}.
The linear growth of $\sigma$ and boundedness of $\sigma'$ and
$\sigma''$ give the coefficient bounds
\eqref{eq:ste_translation_1}--\eqref{eq:ste_translation_5}
for this choice, without a cutoff; the quadratic energy bound
controls the viscosity term. Thus 
$$
\Ex \int_h^T \abs{\,\int_{\T^d} \varphi(x) 
	\bk{\vartheta(u_n(t)) - \vartheta(u_n(t - h))}\,\d x} \d t 
	= O(h^{1/2}), \quad \text{uniformly in $n$},
$$
from Lemma \ref{thm:timetranslation}. 
Hence, Theorem \ref{thm:main_x1} gives us
\begin{align*}
	\int_0^t \int_{\T^d} \varphi \sigma(u_n) \,\d x \,\d W_n 
	\tonweak \int_0^t \int_{\T^d} \varphi \overline{\sigma} \,\d x \,\d W 
	\quad \text{in $L^2(\Omega)$},
\end{align*}
pointwise in $t$. We can easily turn this 
into weak convergence in $L^2(\Omega \times [0,T])$ 
by a uniform-in-$n$ $L^2(\Omega \times [0,T])$ bound. 
By Lemma \ref{thm:wkcv_ae} we find that 
there is a version of $u$, still denoted by $u$ 
such that the following equation holds 
weakly in $x$, $(\omega, t)$-a.e.:
\begin{align}\label{eq:weaklimit_equ_u}
	\d u + {\rm div} \bk{bu}\,\d t = f\,\d t + \overline{\sigma} \,\d {W}.
\end{align}
In order to show that $u$ is a weak solution, 
what is required now is that 
$$
\overline{\sigma} = \sigma(u), \text{ a.e.~in 
$\Omega \times [0,T] \times \T^d$}.
$$
This will be a by-product of showing 
that $\eta(u) = \overline{\eta}$ a.e. 
(and therefore of the strong convergence 
$u_n \to u$). 

We can mollify \eqref{eq:weaklimit_equ_u} 
and derive an analogue of \eqref{eq:renorm_semilinear2} 
where $\vartheta = \eta$, and the subscripts $n$ and 
the dissipation terms are absent. 
The coefficients $(b, {\rm div}\,(b), f)$ reside 
in the spaces of convergence 
of each of the coefficients $(b_n, {\rm div} (b_n), f_n)$.
It is then possible using these inclusions to 
take the $\delta \downarrow 0$ 
limit by the dominated convergence 
theorem to get the $\mathbb{P}\otimes \d t$-a.e.~equality:
\begin{equation}\label{eq:renorm_semilinear3}
\begin{aligned}
& \int _{\T^d} \psi(x) \bk{\eta(u(t)) - \eta(u_0)}\,\d x
- \int_0^t \int_{\T^d}\psi \eta'(u) f\,\d x \,\d s \\
&\quad\,\, +  \int_0^t \int_{\T^d }\psi \big[ {\rm div}\,(b)\bk{\eta'(u) u 
	- \eta(u)} + {\rm div}\,(b  \eta(u))\big]\,\d x \,\d s \\
& =\int_0^t \int_{\T^d} \psi \eta'(u) \overline{\sigma} \,\d x  \,\d W 
+\frac12  \int_0^t \int_{\T^d} \psi \eta''(u)\abs{ \overline{\sigma}}^2\,\d x \,\d s.
\end{aligned}
\end{equation}
In particular, the stochastic integral term 
attains the appropriate limit by the 
standard \cite[Lemma 2.1]{Debussche:2011aa} 
as $\int_{\T^d} \psi \eta'(u_\delta)
\overline{\sigma}_\delta\,\d x
\to \int_{\T^d} \psi \eta'(u)\overline{\sigma}\,\d x$ 
in $L^2([0,T])$ a.s., as in \eqref{eq:temam1} above.

We now compare $\overline{\eta}$ and 
$\eta(u)$ by subtracting 
the integrated form of 
\eqref{eq:renorm_semilinear3} (with $\psi \equiv 1$) 
from \eqref{eq:u_n_limit_equ_ste}. 
This gives us:
\begin{equation}\label{eq:gronwall1_ste}
\begin{aligned}
&\int_{\T^d} \overline{\eta}(t) - \eta(u(t)) \,\d x 
+ \int_0^t \int_{\T^d} 
	\Big(\overline{\bk{\eta' u - \eta}} 
		- \bk{\eta'(u)u - \eta(u)} \Big) \,{\rm div}\,(b) \,\d x\,\d s\\
&\qquad\le  \int_0^t \int_{\T^d} 
	\bk{\overline{\eta'} - \eta'(u)} f\,\d x \,\d s
		+ \frac12\int_0^t \int_{\T^d} 
		{\overline{\eta'' \abs{\sigma}^2}
			- \eta''(u) \abs{\overline{\sigma}}^2} \,\d x \,\d s\\
&\qquad\qquad + \int_0^t \int_{\T^d} 
		{\overline{\eta' \sigma} 
		- \eta'(u) \overline{\sigma}} \,\d x \,\d W.
\end{aligned}
\end{equation}

Using $\eta(v) =  \frac12 v^2$, 
\begin{align}\label{eq:gronwall2_ste}
\overline{\bk{\eta' u - \eta}} 
		- \bk{\eta'(u)u - \eta(u)} 
		 = \overline{\eta } -\eta(u), 
\qquad
{\overline{\eta'} - \eta'(u)} = u - u = 0.
\end{align}

Since $\eta'' = 1$, 
\begin{align*}
{\overline{\eta'' \abs{\sigma}^2}
			- \eta''(u)\abs{ \overline{\sigma}}^2} 
&= {\overline{ \abs{\sigma}^2}
			-\abs{ \overline{\sigma}}^2}
=  \int_{\R} \abs{\sigma(v) - \overline{\sigma}}^2 
		\,\nu_{\omega, t, x}(\d v), 
\end{align*}
where $\nu_{\omega, t ,x}$ is the Young 
measure at $(\omega ,t ,x )$ characterising 
the weak convergence of $u_n \rightharpoonup u$ 
\cite[Theorem 6.2]{Ped1997}. 
For any $Z \in \bk{L^2(\nu_{\omega, t, x})}^k$, 
we have the elementary inequality for 
the variance:
\begin{align*}
&\int_{\R} \abs{Z(v) - \int_{\R} Z(w) 
	\,\nu_{\omega, t, x}(\d w)}^2\, \nu_{\omega, t, x}(\d v) \\ 
&\qquad\quad= \int_{\R} \abs{Z}^2(v) \nu_{\omega, t, x}(\d v) 
	- \abs{ \int_{\R} Z(v) \nu_{\omega, t, x}(\d v) }^2
\le  \int_{\R} \abs{Z}^2(v) \nu_{\omega, t, x}(\d v).
\end{align*}
Let $L$ be the Lipschitz constant of $\sigma$. 
Applying the variance inequality with 
$Z(v) = \sigma(v) - \sigma(u(\omega, t, x))$, 
and using $\int_{\R} v \,\nu_{\omega, t, x}(\d v)
	 = u(\omega, t, x)$, we get
\begin{equation}\label{eq:gronwall3_ste}
\begin{aligned}
\int_{\R} \abs{\sigma(v) - \overline{\sigma}}^2 
	&\,\nu_{\omega, t, x}(\d v) 
 \le \int_{\R} \abs{\sigma(v) - \sigma(u(\omega, t, x))}^2 
 	\,\nu_{\omega, t, x}(\d v)\\
& \le L^2 \int_{\R} \abs{v - u(\omega, t, x)}^2 
	\,\nu_{\omega, t, x}(\d v) 
= 2L^2\big( \overline{\eta} - \eta(u)\big).
\end{aligned}
\end{equation}

Putting \eqref{eq:gronwall2_ste} and 
\eqref{eq:gronwall3_ste} back in 
\eqref{eq:gronwall1_ste}, and taking 
an expectation, we find:
\begin{equation}\label{eq:gronwall4_ste}
\begin{aligned}
\Ex \int_{\T^d} \overline{\eta}(t) - \eta(u(t)) \,\d x 
+ \Ex \int_0^t \int_{\T^d} 
	\big( \overline{\eta} &- \eta(u)  \big) \,{\rm div}\,(b) \,\d x\,\d s\\
&\le  L^2 \Ex \int_0^t \int_{\T^d} 
		 \overline{\eta} - \eta(u) \,\d x \,\d s.
\end{aligned}
\end{equation}
By convexity, we already know that 
$\overline{\eta} - \eta(u) \ge 0$. 
Gronwall's lemma applied 
to the inequality above then tells us 
that $\overline{\eta} = \eta(u)$, 
$\mathbb{P}\otimes \d t \otimes \d x$-a.e. 

The equality
$\overline\eta=\eta(u)$ with $\eta(v)=v^2/2$ gives
$$
\int_\R\abs{v-u(\omega,t,x)}^2\,
\nu_{\omega,t,x}(\d v)=0,
$$
so the Young measure is $\delta_{u(\omega,t,x)}$.
For strong convergence, write
$Q=\Omega\times(0,T)\times\T^d$, with its finite product measure.
Since $p>2$, weak convergence of $u_n^2$ in $L^{p/2}(Q)$ to $u^2$
gives $\norm{u_n}_{L^2(Q)}^2\to\norm{u}_{L^2(Q)}^2$.
Weak convergence of $u_n$ in $L^p(Q)$ also implies weak convergence
in $L^2(Q)$. Hence
$$
\norm{u_n-u}_{L^2(Q)}^2
=\norm{u_n}_{L^2(Q)}^2-2\int_Q u_nu+\norm{u}_{L^2(Q)}^2
\longrightarrow0.
$$
For $2<q<p$, choose $0<\theta<1$ so that
$1/q=\theta/2+(1-\theta)/p$. The uniform $L^p(Q)$ bound yields
$$
\norm{u_n-u}_{L^q(Q)}
\le\norm{u_n-u}_{L^2(Q)}^\theta
\norm{u_n-u}_{L^p(Q)}^{1-\theta}\longrightarrow0.
$$
For $1\le q\le2$, the same conclusion follows from H\"older's
inequality on $Q$. Every subsequence has a further subsequence with
this conclusion and the same weak limit $u$, so the full sequence
converges strongly for every $1\le q<p$.
The Lipschitz bound
$\norm{\sigma(u_n)-\sigma(u)}_{L^2(Q)}
\le L\norm{u_n-u}_{L^2(Q)}$ identifies
$\overline\sigma=\sigma(u)$.
This proves the theorem.
\end{proof}

\begin{remark}\label{rem:pf_artifact}
The cutoff argument above removes the extra bound on
$\sigma''$ in \eqref{eq:sigma_technical}.
To remove the higher initial-data moment, use a convex, linearly
growing
approximation ${\eta}_\ell$ 
with bounded first and second derivatives 
(see,e.g., \cite[Equation (4.2)]{GHKP-inviscid}) 
in place of $\eta$ to derive inequalities for 
the weak limits $\overline{{\eta}_\ell}$.
Choose $\eta_\ell=\eta$ on $[-\ell,\ell]$, with
$0\le\eta_\ell''\le1$ and $\eta_\ell''=0$ outside $[-2\ell,2\ell]$.
This function will give us
$({\eta}_\ell)'\sigma \in C^{1,1}(\R;\R^k)$, which is  
needed in order to derive bounds analogous to 
\eqref{eq:ste_translation_1} -- \eqref{eq:ste_translation_5}.
For fixed $\ell$, $\eta_\ell'\sigma$ has at most linear growth,
so Theorem~\ref{thm:main_x2} uses only the original $p$-moment.
We can then derive the following 
in place of \eqref{eq:gronwall1_ste}:
\begin{align*}
	&\int_{\T^d} \overline{{\eta}_\ell}(t) - \eta(u(t)) \,\d x 
	+ \int_0^t \int_{\T^d} 
	\Big(\overline{\bk{\bk{{\eta}_\ell}' u - {\eta}_\ell}} 
	- \bk{\eta'(u)u - \eta(u)} \Big) \,{\rm div}\,(b) \,\d x\,\d s\\
	&\le  \int_0^t \int_{\T^d} \bk{\overline{\bk{{\eta}_\ell}'} - \eta'(u)} f\,\d x\,\d s
	+ \frac12\int_0^t \int_{\T^d} 
	{\overline{\bk{{\eta}_\ell}'' \abs{\sigma}^2}
	- \eta''(u) \abs{\overline{\sigma}}^2} \,\d x \,\d s\\
	&\qquad + \int_0^t \int_{\T^d} 
	{\overline{\bk{{\eta}_\ell}' \sigma} 
	- \eta'(u) \overline{\sigma}} \,\d x \,\d W.
\end{align*}

It is possible to derive manageable expressions
for ${\eta}_\ell(u_n) - \eta(u_n)$, $\bk{{\eta}_\ell}'(u_n) - \eta'(u_n)$, 
etc.~(see \cite[Remark 7.11]{GHKP-inviscid}), 
which are supported on the small sets 
$\{\abs{u_n} \ge \ell\}$. For $\ell\ge1$, \eqref{eq:p_energybound} and H\"older's inequality give
\begin{align*}
\sup_n\Ex\,\operatorname*{ess\,sup}_{0<t<T}
\int_{\T^d}|u_n|^2\,1_{\{|u_n|>\ell\}}\,\d x
&\le C\ell^{2-p},\\
\sup_n\Ex\,\operatorname*{ess\,sup}_{0<t<T}
\|u_n1_{\{|u_n|>\ell\}}\|_{L^{p'}_x}
&\le C\ell^{2-p}.
\end{align*}
These bounds pass to the weak-limit errors and control the quadratic
terms and their products with $\operatorname{div}b$, and the forcing
error paired with $f$, respectively.
Take expectations at fixed $\ell$.
Localisation and the inherited energy bounds give the martingales
zero expectation. The initial entropy difference is nonpositive
since $\eta_\ell\le\eta$.
Letting $\ell\uparrow\infty$ and applying Gronwall gives
$\eta(u)=\overline\eta$ a.e. 
That will imply strong $L^2_{\omega,t,x}$ convergence and, by the
same interpolation as in Theorem~\ref{thm:ste_semilinear}, strong
convergence in $L^q_{\omega,t,x}$ for every $1\le q<p$. 
Using this linearly growing nonlinearity $\eta_\ell$, 
we can also relax the assumption \eqref{eq:sigma_technical} 
on the initial data to convergence in $L^p_{\omega,x}$, 
where this was used in \eqref{eq:u_0_L2p_bound}.

The endpoint $q_2=2$ requires a separate uniform-integrability
estimate for the quadratic compositions, an $L^1_x$ bound alone does
not suffice. 

\end{remark}

Our convergence theorems 
are important even in the linear, additive noise case. 
We consider a transport equation with 
noise $\sigma_n \,\d W_n$, where $\sigma_n \to \sigma$ 
in $L^p([0,T]\times \T^d;\R^k)$. 
Our convergence theorems are required 
in the renormalised equation for $\d \eta(u_n)$, 
$\eta(v) = \frac12 v^2$, to obtain the limit 
for $\int_0^t \eta'(u_n) \sigma_n\,\d W_n$ 
(cf.~\eqref{eq:eta(u_n)_ste} above and 
see \eqref{eq:eta(u_n)_steadd} below). 
Suitably modifying the definition of 
solutions to take into account the 
additive nature of the noise, we can 
establish the following theorem:

\begin{theorem}\label{thm:ste_add}
Fix $p > 2$. Set $p' = p/\bk{p - 1}$ and 
$p'' = p/\bk{p - 2}$, respectively the 
H\"older conjugates of $p$ and $p/2$. 
Let $\{u_n\}_{n = 1}^\infty$ 
be a sequence of weak solutions to 
\begin{align*}
\d u_n + {\rm div}\,(b_n u_n)\,\d t
  = f_n\,\d t  + \frac1n \Delta u_n\,\d t  
  	+ \sigma_n \,\d{W}_n, 
\quad u_n(0) = u_{0,n},
\end{align*}
for which $u_n\rightharpoonup u$ in 
$L^p(\Omega \times [0,T] \times \T^d)$.
Let $W_n$ and $W$ be Wiener processes with respect to
$(\mathcal F_t^n)$ and $(\mathcal F_t)$, respectively.
Assume that $u_n$ and $u$ admit predictable representatives
for these respective filtrations, and that $u_{0,n}$ and $u_0$
are $\mathcal F_0^n$- and $\mathcal F_0$-measurable, respectively.
Assume also that, for each $n$,
$$
\Ex\,\operatorname*{ess\,sup}_{0<t<T}
\norm{u_n(t)}_{L^p(\T^d)}^p<\infty.
$$
Suppose
\begin{itemize}
\item[(i)]
	$\sigma_n \to \sigma$ in 
	$L^{p}([0,T] \times \T^d;\R^k)$, 

\item[(ii)] $\{\pd_t\sigma_n\}_{n\ge1}
	\subset_b L^1([0,T];L^{p'}(\T^d;\R^k))$, and\\
	$\{\sigma_n\}_{n\ge1}
	\subset_b L^\infty([0,T];W^{1,p}(\T^d;\R^k))$,
\item[(iii)]
	$\{b_n\}_{n \ge 1} \subset_b 
		L^1([0,T];W^{1,p''}(\T^d;\R^d))$, and \\
	$\{{\rm div}\,(b_n)\}_{n \ge 1}\subset_b 
		L^1([0,T];L^\infty(\T^d))$,  
\item[(iv)] $b_n \to b$ in $L^1([0,T];L^{p'}(\T^d;\R^d))$, and \\
	${\rm div}\,(b_n) \to {\rm div}\,(b)$, in 
	$L^1([0,T];W^{1,p''}(\T^d))$,
\item[(v)] $f_n \to f$ in $L^1([0,T];L^{p}(\T^d))$, 
$u_{0,n} \to u_0$ in $L^{p}(\Omega \times \T^d)$, and 
\item[(vi)]
	$W_n \to W$ a.s.~in $C([0,T];\R^k)$.
\end{itemize}
Then $u$ is a weak solution of
$$
\d u + {\rm div}\,(b u)\,\d t  = f\,\d t  + \sigma \,\d{W}, \quad u(0) = u_0,$$
and 
$u_n \to u$ strongly in $L^q(\Omega \times [0,T] \times \T^d)$
for every $1\le q<p$. 
\end{theorem}

\begin{proof}
We shall avail ourselves of the same 
strategy we used in the proof of Theorem 
\ref{thm:ste_semilinear}, and borrow 
heavily from the machinery described there.
The same quadratic truncation argument as in
\eqref{eq:p_energybound}, with $\sigma(u_n)$ replaced by $\sigma_n$,
gives the uniform energy bound for $u_n$. 
Let $\eta(v) = \frac12 v^2$ again. 
By arguing as in \eqref{eq:eta(u_n)_ste}, 
we get the analogous inequality
\begin{equation}\label{eq:eta(u_n)_steadd}
\begin{aligned}
	\int_{\T^d} &\eta(u_n(t)) -  \eta(u_{0,n})\,\d x 
	+\int_0^t \int_{\T^d} \bk{\eta'(u_n) u_n - \eta(u_n)}\,
	{\rm div}\bk{b_n} \,\d x \,\d s \\
	& \le\int_0^t  \int_{\T^d} \eta'(u_n) f_n 
	+ \frac12 \abs{\sigma_n}^2 \,\d x \,\d s 
	+ \int_0^t \int_{\T^d}
	 \eta'(u_n) \sigma_n\,\d x \,\d W_n.
\end{aligned}
\end{equation}

We now seek to take the limit as $n \to \infty$ 
to get the equation for $\overline{\eta}$. 
By the bound \eqref{eq:p_energybound} on $u_n$, 
by the form of $\eta$ and the uniqueness 
of weak limits, we have the following weak convergences 
as $n \uparrow \infty$:
\begin{align*}
\bk{\eta(u_n) , \eta'(u_n)u_n,
	\eta'(u_n)}
&	\rightharpoonup \bk{\overline{\eta},
	 2 \overline{\eta}, u  }
	\text{ in $(L^{p/2}(\Omega \times [0,T]\times \T^d))^3$},\\
	\eta'(u_n) \sigma_n &\rightharpoonup u \sigma 
	\text{ in $L^{p/2}(\Omega \times [0,T]\times \T^d;\R^k)$},
\end{align*}along an 
unrelabelled subsequence. 
The convergence of the final element 
of the tuple follows from the 
strong convergence of  $\sigma_n$ 
in $L^p([0,T]\times \T^d;\R^k)$, assumed in (i).

The convergence of \eqref{eq:eta(u_n)_steadd} to 
\begin{equation}\label{eq:etabar_steadd}
\begin{aligned}
\int_{\T^d} &\overline{\eta}(t) -  \eta(u_{0})\,\d x 
	+\int_0^t \int_{\T^d} \overline{\eta' u - \eta}\,\,
		{\rm div}\bk{b} \,\d x \,\d s\\
& \le\int_0^t  \int_{\T^d} u\, f 
	+ \frac12 \abs{\sigma}^2 \,\d x \,\d s
	+ \int_0^t \int_{\T^d}
	 u \sigma \,\d x \,\d W,
\end{aligned}
\end{equation}
follows from the assumed convergences 
on the coefficients $b_n$, $f_n$, $\sigma_n$ 
and the weak convergence $u_n \rightharpoonup u$, 
except in the stochastic integral. 

In the stochastic integral 
$\int_0^t \int_{\T^d}u_n \sigma_n \,\d x \,\d W_n$,  
$ \int_{\T^d}u_n \sigma_n \,\d x$ converges 
only weakly in  $(\omega, t)$ due to the assumed 
weak convergence of $u_n$.  The lack of strong 
temporal compactness compels us to apply 
one of the variants of Theorem \ref{thm:main_x1} here. 
Using the supremum bound \eqref{eq:p_energybound} 
on $u_n$, we find that 
\begin{equation}\label{eq:LpL1_bound_steadd}
\begin{aligned}
	\Ex &\norm{\eta'(u_n) 	
	\sigma_n }_{L^p_tL^1_x}^{p}
	\lesssim  \norm{\sigma_n}_{L^p_tL^{p'}_x}^p\lesssim 1, 
	\quad 
	\Ex \sup_{t \in [0,T]}\norm{u_n(t)}_{L^p_x}^p
	 \lesssim 1.
\end{aligned}
\end{equation}

Next we deduce the temporal translation 
estimate using Lemma \ref{thm:timetranslation}. 
By the product formula,
$$
\d \bk{u_n \sigma_n} 
	= \pd_t \sigma_n \,  u_n\,\d t + \sigma_n\,\d u_n.
$$
We can integrate this against 
any $\psi \in C^1(\T^d)$ to get:
\begin{align*}
&\int_{\T^d} \psi(x) \big( \bk{u_n \sigma_n}(t)   	
	-  \bk{u_n \sigma_n}(t - h)\big) \,\d x\\
& = \int_{t - h}^ t \int_{\T^d}
	\underbrace{\psi \pd_t \sigma_n  u_n}_{\tilde{I}_1(n)}\,\d x\,\d r
+ \int_{t - h}^ t \int_{\T^d} 
	\underbrace{\nabla \bk{\psi \sigma_n} \cdot b_n u_n
 		+ \psi  \sigma_n f_n}_{\tilde{I}_2(n)}\,\d x \,\d r\\
& \quad\,\, - \frac1n \int_{t - h}^ t \int_{\T^d} 
	\underbrace{\nabla\bk{\psi \sigma_n}
		\cdot \nabla u_n}_{-n\tilde{I}_3(n)} \,\d x \,\d r
+ \int_{t - h}^ t \int_{\T^d}  
	\underbrace{\psi\,\sigma_n\otimes\sigma_n}_{\tilde{I}_4(n)}\,\d x \,\d W_n.
\end{align*}
Using \eqref{eq:p_energybound} and 
the assumed bounds in the theorem 
statement, we can show via Lemma 
\ref{thm:timetranslation} that 
$$
\Ex \int_h^T \abs{\,\int_{\T^d} \psi(x) \big( \bk{u_n\sigma_n}(t)   	
	-  \bk{u_n \sigma_n}(t - h)\big) \,\d x}\,\d t
	\lesssim h^{1/2}, \quad \text{uniformly in $n$}.
$$

Indeed, H\"older's inequality gives
\begin{align*}
&\norm{\tilde I_1(n)}_{L^1_{\omega,t,x}}
+\norm{\tilde I_2(n)}_{L^1_{\omega,t,x}}\\
&\quad\lesssim_\psi
\Big(\norm{\pd_t\sigma_n}_{L^1_tL^{p'}_x}
+\norm{\sigma_n}_{L^\infty_tW^{1,p}_x}
 \norm{b_n}_{L^1_tL^{p''}_x}\Big)\\
&\qquad\times\Ex\norm{u_n}_{L^\infty_tL^p_x}
+\norm{\sigma_n}_{L^\infty_tL^{p'}_x}
 \norm{f_n}_{L^1_tL^p_x}\lesssim1.
\end{align*}
Here $1/p+1/p''+1/p=1$ for the transport term, and
$L^p_x\subset L^{p'}_x$ since $p>2$.

Assumption~(ii) also bounds
$\nabla(\psi\sigma_n)$ in $L^2_{t,x}$.
The quadratic energy estimate therefore gives
\begin{align*}
\norm{\tilde I_3(n)}_{L^1_{\omega,t,x}}
&\le\norm{\nabla(\psi\sigma_n)}_{L^2_{t,x}}\,
\frac1n\Ex\norm{\nabla u_n}_{L^2_{t,x}}
\lesssim_\psi n^{-1/2}.
\end{align*}
For the stochastic term, the deterministic coefficient satisfies
$$
\norm{\tilde I_4(n)}_{L^2_{\omega,t}L^1_x}
\le\norm{\psi}_{L^\infty_x}T^{1/2}
\norm{\sigma_n}_{L^\infty_tL^2_x}^2\lesssim_\psi1.
$$
These are the bounds needed for Lemma~\ref{thm:timetranslation}.
By choosing $\psi \equiv 1$, 
Theorem  \ref{thm:main_x2} then 
implies the convergence of the stochastic
integral in \eqref{eq:eta(u_n)_steadd} weakly in 
$L^2(\Omega)$ and we get \eqref{eq:etabar_steadd}.

With $\eta(v) = \frac12 v^2$, the derivation 
of the equation for $\eta(u)$ follows the 
same procedure as that leading up to \eqref{eq:renorm_semilinear3}, 
but is more straightforward. In particular, 
the convergence of stochastic integrals 
in this process only relies on 
\cite[Lemma 2.1]{Debussche:2011aa}, 
seeing as 
$\int_{\T^d}\psi \sigma_\delta u_\delta \,\d x 
\to \int_{\T^d}\psi u \sigma\,\d x$ a.s.~in 
$L^2([0,T])$ by the assumed inclusion 
$\sigma \in L^p([0,T]\times \T^d)$ 
and \eqref{eq:LpL1_bound_steadd} 
(cf.~\eqref{eq:temam1}).

We have that for any $\psi \in C^2(\T^d)$,
\begin{equation}\label{eq:renorm_ste_additive}
\begin{aligned}
&\int_{\T^d} \psi(x) \bk{\eta(u) - \eta(u_0)}\,\d x  
	+ \frac12 \int_0^t  \int_{\T^d} \Big[ \psi\,  {\rm div} (b)  {u}^2
	-  \nabla \psi \cdot {b {u}^2}\Big]\,\d x \,\d s \\
& = \int_0^t \int_{\T^d}\psi u f\,\d x \,\d s  
	+ \int_0^t \int_{\T^d} \psi u \sigma \,\d x\,\d W 
 	+\frac12 \int_0^t  \int_{\T^d}\psi  \abs{\sigma}^2\,\d x \,\d s.
\end{aligned}
\end{equation}
\vspace{.2cm}

Subtracting \eqref{eq:renorm_ste_additive} 
(with $\psi \equiv 1$) from 
\eqref{eq:etabar_steadd}, we find
\begin{align*}
\int_{\T^d} &\overline{\eta}(t) - \eta(u(t)) \,\d x 
	+\int_0^t \int_{\T^d} \bk{\overline{\eta' u - \eta} 
		- \frac12 u^2 }
		{\rm div}\bk{b} \,\d x \,\d s 
 \le 0.
\end{align*}
Since $\overline{\eta' u - \eta}
	 = \frac12 \overline{u^2} = \overline{\eta}$, 
we can take an expectation above 
and apply Gronwall's inequality to get
\begin{align*}
\Ex \int_{\T^d} \overline{\eta} - \eta(u) \,\d x = 0, 
	\quad \text{$\d t$-a.e.}
\end{align*}
Since $\overline{\eta} \ge \eta(u)$ 
$(\omega,t,x)$-a.e.~by 
convexity, we arrive at the a.e.~equality 
$\overline{\eta} = \eta(u)$. 
As after \eqref{eq:gronwall4_ste}, equality of the limiting quadratic
norms gives strong convergence in $L^2(\Omega\times(0,T)\times\T^d)$.
Interpolation with the uniform $L^p$ bound then gives
$u_n\to u$ strongly in $L^q(\Omega\times(0,T)\times\T^d)$
for every $1\le q<p$. 

\end{proof}

\begin{remark}
Under suitable assumptions, our techniques are 
applicable to SPDEs with transport noise. 
Specifically, they can assist in facilitating the passage to the 
limit in stochastic integrals of the form 
$\int_0^T\int_{\T^d}  \sigma_n \nabla u_n\,\varphi\,\d x  
\circ \,\d{W}_n$, where $\circ$ refers to 
the Stratonovich integral, and $\sigma_n$ is $\R^{d \times k}$-valued.
\end{remark} 

\section{Stochastic conservation laws}
\label{sec:cl}
In this section, we focus on an application pertaining to 
sequences of stochastic conservation laws \eqref{eq:SCL}, 
as expressed in the form of kinetic equations \eqref{eq:example_spde2}. 
In \eqref{eq:example_spde2}, the driving noise 
processes $W_n$ are $\R^k$-valued Brownian motions. 
The random defect measures $m_n$ in \eqref{eq:example_spde2} 
take values in $M_b^+( [0,T] \times \T^d\times \R)$,
the space of non-negative Radon measures. Moreover, the solutions 
$\chi_n$ to \eqref{eq:example_spde2} trivially satisfy $0\le \chi_n\le 1$.
The interpretation of the kinetic equations 
\eqref{eq:example_spde2} is in the 
It{\^o} sense and is considered a.s.~in a weak formulation over 
the domain $\T^d \times \R$ \cite{Debussche:2010fk}.

Suppose $W_n=W$, with $W$ Brownian and all integrands predictable
in a common filtration. The bounded linear It\^o map then takes weak
$L^2_{\omega,t}$ convergence of the integrands paired with a test
function in $(x,\xi)$ to weak $L^2_\omega$ convergence of their
stochastic integrals at each fixed time. 
This specific case is straightforward and has been 
treated in, e.g., \cite{Debussche:2010fk,Dotti:2018aa,Dotti:2020}. 
In contrast, our study extends to the scenario involving 
a sequence of potentially distinct Brownian motions 
$\{W_n\}_{n \ge 1}$.

\begin{theorem}\label{thm:cl_weakstability}
Fix $ p > 2$. Let $\chi_n$ be the kinetic 
solution to \eqref{eq:example_spde2} 
and suppose that 
\begin{itemize}
\item[(i)]
	$\sigma_n$ are bounded in $W^{3,1}_\loc(\R;\R^k)$,

\item[(ii)] $F_n \to F$ in  $W^{1,1}_\loc(\R;\R^d)$, 
	and $\sigma_n \to \sigma$  in $W^{1,1}_\loc(\R;\R^k)$,

\item[(iii)] $m_n([0,t]) \rightharpoonup m([0,t])$ 
		in $L^1(\Omega \times [0,T];M_b^+(\T^d\times \R))$, in 
		the sense that for every $\varphi \in C^\infty_c(\T^d \times \R)$ 
		and every $Y \in L^\infty(\Omega \times [0,T])$, 
		$$
		\Ex \left[\int_0^T Y \int_{\T^d \times \R} 
			\varphi\, m_n([0,t],\d x,\d \xi)\,\d t \right] 
		\to \Ex \left[\int_0^T Y \int_{\T^d \times \R} 
			\varphi\, m([0,t],\d x,\d \xi)\,\d t\right], 
		$$
and
$$
\sup_{n\ge1}\Ex\,m_n([0,T]\times\T^d\times\R)<\infty.
$$
\item[(iv)] $\chi_{0,n} \rightharpoonup \chi_0$ in 
		$L^p(\Omega \times \T^d;\mathcal{D}'(\R))$, and 

\item[(v)] $W_n \to W$ in $C([0,T];\R^k)$ a.s.
\end{itemize}
Assume $\chi_n \xrightharpoonup{\star} \chi$ in 
$L^\infty(\Omega \times [0,T] \times \T^d \times \R)$. 
Let $W_n$ and $W$ be Wiener processes with respect to
$(\mathcal F_t^n)$ and $(\mathcal F_t)$, respectively.
Assume that $\chi_n$ and $\chi$ admit predictable representatives
for these respective filtrations, and that $\chi_{0,n}$ and $\chi_0$
are $\mathcal F_0^n$- and $\mathcal F_0$-measurable, respectively.
Then $\chi$ satisfies 

\begin{equation}\label{eq:conservation3}
\begin{aligned}
	\d \chi +\big( F'(\xi) & \cdot \nabla_x \chi 
	 - \pd_{\xi} m \big)\,\d t
\mathbin{+} \sigma(\xi) \pd_\xi \chi \,\d{W}
	- \frac12  \pd_{\xi} \bk{\abs{\sigma(\xi)}^2 \,\pd_{\xi} \chi}\,\d t=0,
\end{aligned}
\end{equation}
a.s.~ in the sense of It\^o for a.e.~$t \in [0,T]$, and in 
$\mathcal{D}'(\T^d \times \R)$ with $\chi(0) = \chi_0$. 
\end{theorem}

By $m_n([0,t])$, we mean $\int_{r = 0}^t m_n(\d r,\d x,\d \xi)$, 
which remains a measure in the space $L^1_\omega (M_b^+)_{x,\xi}$ 
as $m_n$ admits a disintegration \cite[Theorem 1.10]{evans:1990}. 

\begin{proof}
Let $\varphi \in C^\infty_c(\T^d \times \R)$ so that 
$\supp(\varphi) \subset K \Subset \T^d \times \R$. 
Using the definition of weak solutions to  
\eqref{eq:example_spde2} tested against the 
$\varphi$, we see that each term of \eqref{eq:example_spde2} 
tends weakly in $(\omega,t)$ to the appropriate limit by the 
assumptions of the theorem. It remains then 
to argue that the stochastic integral converges, 
whereupon we can invoke Lemma \ref{thm:wkcv_ae} 
to conclude.

Since $\chi_n$ converges only weakly-$\star$ 
in $L^\infty_{\omega, t, x,\xi}$, we will 
apply Theorem \ref{thm:main_x1} to get weak 
convergence of the stochastic integral in $L^2(\Omega)$: 
\begin{align}\label{eq:cl_itoconverge}
	 \int_0^t  \int_{\R}\int_{\T^d}
 	 \pd_{\xi}\bk{\varphi \sigma_n} \chi_n\,\d x \,\d \xi \,\d W_n
	\tonweak   \int_0^t  \int_{\R}\int_{\T^d}
 	 \pd_{\xi}\bk{\varphi \sigma} \chi\,\d x \,\d \xi \,\d W.
\end{align}
By $W^{3,1}_\loc \hookrightarrow L^\infty_\loc$
and the assumption on $\sigma_n$,
and the $L^\infty_{\omega, t, x, \xi}$ bound on 
$\chi_n$, we find
\begin{equation}\label{eq:p_energybound5}
\begin{aligned}
	&\Ex \norm{\pd_{\xi}\bk{\varphi \sigma_n} 
	\chi_n}_{L^p( [0,T] ;L^2(\T^d \times \R;\R^k))}^p
	\lesssim 1.
\end{aligned}
\end{equation}
The time translation estimate will follow 
from Lemma \ref{thm:timetranslation}, which  
can be easily modified to accommodate 
integration on $\T^d \times \R$ instead of $\T^d$, 
against $\xi$-compactly supported test functions. 
Integrating  \eqref{eq:example_spde2} against 
$\psi \pd_{\xi}\bk{\varphi \sigma_n}$, where 
$\psi \in C^2_c(\T^d \times \R)$, set
\begin{align*}
	I_1(n) &: = 
	\nabla_x \otimes \big( \psi\pd_{\xi}\bk{\varphi \sigma_n}   \big) 
	 F'_n(\xi) \chi_n, \quad
	I_2(n) : = 
	 \pd_\xi \bk{\psi \pd_{\xi}\bk{\varphi \sigma_n} }
	 \, m_n, \\ 
	I_3(n) &: = 
	\pd_\xi \bk{\abs{\sigma_n}^2\, \pd_{\xi}
	\bk{\psi\, \pd_{\xi}\bk{\varphi \sigma_n} }}
	 \chi_n, \quad 
	I_4(n) : = 
	\pd_\xi \big({ \sigma_n\otimes\psi\, \pd_{\xi}\bk{\varphi \sigma_n} }\big)
	 \chi_n. 
\end{align*}
As in \eqref{eq:ste_translation_1} -- \eqref{eq:ste_translation_5}, 
we can estimate these terms separately. We have 
\begin{align*}
	\norm{I_1(n)}_{L^1_{\omega, t, x, \xi}} 
	\le T \norm{\nabla  \otimes\bk{\psi\,\pd_\xi
	\bk{\varphi \sigma_n}}}_{L^\infty(\T^d \times \R)} 
	\norm{F'_n}_{L^1(K)} \lesssim 1,
\end{align*}
since $\sigma_n \in W^{3,1}_\loc(\R;\R^k) 
	\hookrightarrow W^{1,\infty}_\loc(\R;\R^k)$, 
and $F_n \to F$ in $W^{1,1}_\loc(\R;\R^d)$. The mass bound in (iii) gives
\begin{align*}
\norm{I_2(n)}_{L^1_\omega {TV}_{t,x,\xi}}
&\le\norm{\pd_\xi\bk{\psi\pd_\xi(\varphi\sigma_n)}}_{L^\infty_{x,\xi}}\,
\Ex\,m_n([0,T]\times\T^d\times\R)
\lesssim1.
\end{align*}
The coefficient is uniformly bounded on the fixed compact support
by (i).

Next, using that $0\le \chi_n\le 1$,
\begin{align*}
	\norm{I_3(n)}_{L^1_{\omega, t, x, \xi}} 
	&\le
	T\norm{\pd_\xi \bk{ \abs{\sigma_n}^2\, \pd_{\xi}
	\bk{\psi\, \pd_{\xi}\bk{\varphi \sigma_n} }}}_{L^1(K)}
	\lesssim 1.
\end{align*}
Finally, we have 
\begin{align*}
	&\norm{I_4(n)}_{L^2_{\omega,t}L^1_{x,\xi}} 
	 \le T^{1/2}\norm{\pd_\xi \big({ \sigma_n \otimes \psi\, \pd_{\xi}
	\bk{\varphi \sigma_n} }\big)}_{L^1_{x,\xi}} \lesssim 1.
\end{align*}

The bounds above are $n$-independent. Therefore, 
Lemma \ref{thm:timetranslation} gives us
\begin{align*}
	\Ex\int_h^T\abs{ \int_{\R}\int_{\T^d} \psi
	\pd_\xi\bk{\varphi \sigma_n}
	 \big({\chi_n(t)  - \chi_n(t - h)}\big)\,\d x \,\d \xi} \,\d t
	& \lesssim_{T,\psi} h^{1/2},
\end{align*}
uniformly in $n$. Along with \eqref{eq:p_energybound5} 
and the weak convergence of $\chi_n$ on compact $\xi$-sets,
together with the strong local convergence of the deterministic
coefficients $\pd_\xi(\varphi\sigma_n)$, Theorem \ref{thm:main_x1} allows us to conclude 
that the stochastic integral term converges as in \eqref{eq:cl_itoconverge}.

It then follows that $\chi$ is a weak solution 
to the limiting equation \eqref{eq:conservation3}.
\end{proof}

\vspace{.2cm}

{\bf Declaration of generative AI use. }The authors report generative AI was not used in their research or preparation of this manuscript.

\end{document}